\documentclass[10pt]{amsart}
\usepackage{amsthm}
\usepackage{amstext}
\usepackage{amssymb}
\usepackage[alphabetic]{amsrefs}  

\numberwithin{equation}{section}
\numberwithin{figure}{section}

\newtheorem{theorem}{Theorem}[section]
\newtheorem{lemma}[theorem]{Lemma}
\theoremstyle{definition}
\newtheorem{definition}[theorem]{Definition}
\newtheorem{example}[theorem]{Example}

\theoremstyle{remark}
\newtheorem{remark}[theorem]{Remark}
\numberwithin{equation}{section}
\theoremstyle{plain}

\newtheorem{corollary}[theorem]{Corollary}
\newtheorem{observation}[theorem]{Observation}

\newtheorem{proposition}[theorem]{Proposition}

\newcommand{\R}{\mathbb{R}}

\begin{document}

\title[Conformal diffeomorphisms of GQE manifolds]{Conformal diffeomorphisms of gradient Ricci solitons and generalized quasi-Einstein manifolds}
\author{Jeffrey L. Jauregui}
\address{209 S. 33rd Street,
David Rittenhouse Lab,
Dept. of Math,
Philadelphia, PA 19104.}
\email{jauregui@math.upenn.edu}
\urladdr{http://www.math.upenn.edu/~jjau}
\author{William Wylie}
\address{215 Carnegie Building\\
Dept. of Math, Syracuse University\\
Syracuse, NY, 13244.}
\email{wwylie@syr.edu}
\urladdr{http://wwylie.mysite.syr.edu}
\thanks{The second author was supported in part by NSF-DMS grant 0905527. }

\begin{abstract}
In this paper we extend some well-known rigidity results for conformal changes of Einstein metrics to 
the class of generalized quasi-Einstein (GQE) metrics, which includes gradient Ricci solitons.  In order to do so, we introduce the 
notions of conformal diffeomorphisms and vector fields  that preserve a GQE structure.   We show that 
a complete GQE metric admits a structure-preserving, non-homothetic complete conformal vector field   
if and only if it is a round sphere.  We also classify the structure-preserving conformal diffeomorphisms.  
In the compact case, if a GQE metric admits a structure-preserving, non-homothetic conformal diffeomorphism,  then the metric is conformal to the sphere, and isometric to the sphere in the case of a 
gradient Ricci soliton.  In the complete case, the only  structure-preserving non-homothetic  conformal diffeomorphisms from
a shrinking or steady gradient Ricci  soliton to another soliton are the conformal transformations of spheres and stereographic projection. 
\end{abstract}


\maketitle

\section{Introduction}
It is well-known that the Einstein condition on a Riemannian manifold  is not conformally invariant.  In  the 1920s 
Brinkmann \cite{Brinkmann} classified when two Einstein metrics are conformal to each other and  
Yano--Nagano \cite{YanoNagano}  later proved that if a complete  Einstein metric admits a complete conformal field then it is 
a round sphere.  For further results in this direction  see \cite{Nagano}, \cite{Lichnerowicz} and pp. 
309--311 of \cite{KN}.  For the pseudo-Riemannian case and many more references, see \cite{KR09}.   

In this paper we show that these results have natural extensions to the class of  \emph{generalized quasi-Einstein} (GQE) metrics, that is, Riemannian metrics $g$ on a manifold $M$ of dimension $n \geq 3$ satisfying
\begin{equation}
\label{eqn_GQE}
\mathrm{Ric} + \mathrm{Hess} f + \alpha df \otimes df = \lambda g
\end{equation}
for some smooth functions $f, \alpha, \lambda$ on $M$, where $\mathrm{Ric}$ and $\mathrm{Hess}$ are the Ricci
curvature and Hessian with respect to $g$.  GQE manifolds\footnote{We note that this class of metrics differs from 
the K\"{a}hler generalized quasi-Einstein metrics of Guan \cite{Guan} and the generalized quasi-Einstein 
metrics of Chaki \cite{Chaki}.} were recently introduced by Catino \cite{Catino}, 
who proved a local classification of GQE metrics
with divergence-free Weyl tensor.  GQE metrics generalize:
\begin{itemize}
 \item Einstein metrics: $\mathrm{Ric} = \lambda g$ where $\lambda \in \mathbb{R}$,
 \item gradient Ricci solitons: $\mathrm{Ric} + \mathrm{Hess} f = \lambda g $, where $\lambda \in \mathbb{R}$,
 \item gradient Ricci \emph{almost} solitons: $\mathrm{Ric} + \mathrm{Hess} f = \lambda g $, where $\lambda \in C^{\infty}(M)$, 
 introduced by Pigola--Rigoli--Rimoldi--Setti \cite{PRRS}, and 
 \item $m$-quasi-Einstein metrics: $\alpha = - \frac{1}{m}$ for a positive integer $m$ and $\lambda  \in \mathbb{R}$,
introduced by Case--Shu--Wei \cite{CSW}; these include the static metrics when $m=1$.
\end{itemize}

We will consider diffeomorphisms  between GQE manifolds  that preserve the structure in the following sense. 

\begin{definition} A diffeomorphism  $\phi$ from a GQE manifold $(M_1, g_1, f_1, \alpha_1, \lambda_1)$
to a GQE manifold $(M_2, g_2, f_2, \alpha_2, \lambda_2)$ is said to \emph{preserve the GQE structure} if
$\phi^* \alpha_2 = \alpha_1$ and $\phi^* df_2 = df_1$.   A vector field $V$  on a GQE manifold  \emph{preserves the GQE structure} if $D_V \alpha = 0$ and
$D_V f$ is constant, or equivalently, if the local flows of $V$  preserve the GQE structure.\end{definition}

A \emph{conformal diffeomorphism} $\phi$  between Riemannian manifolds $(M_1,g_1)$ and $(M_2,g_2)$   is a diffeomorphism such that  
\[ \phi^*g_2 = w^{-2} g_1\]
 for some function $w>0$ on $M_1$.  A \emph{conformal vector field} on a Riemannian manifold
$(M,g)$ is a vector field whose local flows are conformal diffeomorphisms; equivalently, $V$ satisfies 
 \[ L_V g = 2 \sigma g \]  for some function $\sigma$ on $M$, where $L$ is the Lie derivative.
 
 A trivial example of a conformal diffeomorphism  that preserves the GQE structure is any 
homothetic rescaling ($\phi$ = identity, $g_2=c^2 g_1$).
We will say a conformal diffeomorphism is \emph{non-homothetic} if  $w$ is not constant.  Similarly, a conformal vector field is non-homothetic if $\sigma$ is not constant.   

We show that  conformal diffeomorphisms and vector fields that preserve a 
GQE structure only exist in very rigid situations.  Our most general result is the following classification theorem for non-homothetic conformal transformations that preserve 
a generalized quasi-Einstein structure.   This result holds in both the local and global settings.
\begin{theorem}
\label{thm_local_classification}
Let $\phi$ be a non-homothetic, structure-preserving conformal diffeomorphism between GQE manifolds $(M_1,g_1,f_1,\alpha_1,\lambda_1)$
and $(M_2,g_2,f_2,\alpha_2,\lambda_2)$ of dimension $n \geq 3$.
Then, about points where $\alpha_i \neq \frac{1}{n-2}$, $g_1$ and $g_2$ are both of the form
\begin{equation}
\label{eqn_g_f(t)}
g_i = ds^2 + v_i(s)^{2} g_{N},
\end{equation}
where $(N,g_N)$ is an $(n-1)$-manifold  independent of $s$ and $f_i=f_i(s)$, or 
\begin{equation}
\label{eqn_g_f(x)}
g_i = e^{\frac{2f_i}{n-2}}\left( ds^2 + v_i(s)^{2} g_{N} \right) ,
\end{equation}
where $(N,g_{N})$ is an $(n-1)$-manifold independent of $s$ and 
$f_i$ is a function on $N$. 
If either $g_1$ or $g_2$ is complete and $\alpha_i\neq \frac{1}{n-2}$, then 
the metrics are globally either of the form (\ref{eqn_g_f(t)}) or (\ref{eqn_g_f(x)}). Moreover, 
if $n \geq 4$ or  $\alpha$ is constant in case (\ref{eqn_g_f(t)}), then $g_N$ is Einstein;
in case (\ref{eqn_g_f(x)}), $g_N$ is conformal to a GQE manifold with potential $f_i$. 
Finally, only (\ref{eqn_g_f(t)}) is possible if $n=3$.
\end{theorem}

\begin{remark} If $\alpha_1 \equiv \frac{1}{n-2}$, then $g_1$ is conformal to an Einstein 
metric (see Proposition \ref{prop_Ric_h}); these spaces fall into the Einstein case studied by Brinkmann \cite{Brinkmann}.\end{remark}

\begin{remark} When $f$ is constant, cases (\ref{eqn_g_f(t)}) and (\ref{eqn_g_f(x)}) are the same. The metric $g_2$ need not be complete if $g_1$ 
is, even in the Einstein case: stereographic projection provides a counterexample.
\end{remark}

\begin{remark} We give examples in sections \ref{sec_f(t)} and \ref{sec_f(x)} showing both cases
in Theorem \ref{thm_local_classification} may occur. We also show that the two cases do not occur on the same connected 
manifold unless $f$ is constant.
\end{remark}

In the compact case, we further obtain the following. 
\begin{theorem}
\label{thm_compact_case}
Let $\phi$ be a non-homothetic, structure-preserving conformal diffeomorphism between compact GQE manifolds $(M_1,g_1,f_1,\alpha_1,\lambda_1)$
and $(M_2,g_2,f_2,\alpha_2,\lambda_2)$.
Then $(M_i, g_i)$ are conformally diffeomorphic to the standard round metric on $S^n$.  Moreover, if $\alpha_1 \neq \frac{1}{n-2}$, 
then $(M_i,g_i)$ are rotationally-symmetric. 
\end{theorem}

The case of conformal fields exhibits greater rigidity than the case of discrete conformal changes.  For instance, we prove:
\begin{theorem}
\label{thm_compact_case_conformal_fields}
Suppose $(M,g,f,\alpha,\lambda)$ is a complete GQE manifold, with $\alpha \neq \frac{1}{n-2}$, that admits a structure-preserving non-homothetic 
conformal field: $L_V g = 2\sigma g$.  If $\sigma$ has a critical point (e.g., if $M$ is compact),
then $f$ is constant and $(M,g)$ is isometric to a simply-connected space form.
\end{theorem}

Moreover, the round sphere is the only possibility if   the conformal field is assumed to be complete,
generalizing Yano--Nagano's result.
\begin{theorem} \label{thm_complete_case_conformal_fields} If a complete GQE manifold $(M,g,f,\alpha,\lambda)$  with $\alpha \neq \frac{1}{n-2}$ admits a non-homothetic 
complete conformal field $V$ preserving the GQE structure then $f$ is constant and $(M,g)$ is isometric to a round sphere. 
\end{theorem}

In fact, we obtain a full local classification without the completeness assumption on $V$ nor $g$. There are several 
examples; we delay further discussion to section \ref{sec_conf_fields}. 

We also obtain more rigidity in the case of a gradient Ricci soliton (i.e. $\alpha=0$ and $\lambda$ is constant).  
A gradient Ricci soliton  $(M,g,f)$ is called \emph{shrinking}, \emph{steady},  or \emph{expanding} depending on whether  $\lambda>0$, $\lambda = 0$, $\lambda <0$ respectively.   Combining our results with some other well-known results for gradient Ricci solitons gives us  the  following theorem. 
\begin{theorem} 
\label{thm_solitons}
Let $\phi$ be a non-homothetic, structure-preserving conformal diffeomorphism between 
GQE manifolds $(M_1,g_1,f_1,0,\lambda_1)$ and $(M_2,g_2,f_2,0,\lambda_2)$,
and assume $(M_1,g_1,f_1)$ is a complete gradient Ricci soliton. Then $g_1$ and $g_2$ are both metrics of the form 
(\ref{eqn_g_f(t)}), and:
\begin{itemize} 
\item If $M_1$ is compact, then $g_1$ and $g_2$ are both round metrics on the sphere. 
\item  If $(M_1, g_1, f_1)$ is either shrinking or steady, then it is a round metric on the sphere, the flat metric on 
$\mathbb{R}^n$, the Bryant soliton, or a product $\mathbb{R} \times N$ where $N$ is Einstein with Einstein constant $\lambda$.
\item   If, in addition, $(M_2, g_2, f_2)$ is a soliton,
then either $(M_i,g_i)$ are round metrics on the sphere or $\phi$ is a stereographic projection with $(M_1,g_1)$ flat Euclidean 
space and  $(M_2,g_2)$ a round spherical metric with a point removed.
\item  If $(M,g,f)$ is a complete gradient Ricci soliton admitting a non-homothetic conformal field that preserves the structure, 
then $(M,g)$ is Einstein and $f$ is constant. 
\end{itemize}
\end{theorem}

\begin{remark}
In the last case, note that complete Einstein metrics admitting non-homothetic conformal fields
were classified by Kanai \cite{Kanai}; we recall this classification in remark \ref{rmk_kanai}.
\end{remark}

\begin{remark} We also obtain that  $g_1$ and $g_2$  are both of the form (\ref{eqn_g_f(t)}) when $g_1$ is $m$-quasi-Einstein.  $m$-quasi-Einstein metrics of the form  (\ref{eqn_g_f(t)}) are found in   \cite{Boehm} (cf. \cite{HPW}).   Examples of complete expanding  gradient Ricci solitons of the form  (\ref{eqn_g_f(t)}) are found in  \cite{Chowetc}. 
\end{remark}

\begin{remark} Interesting results for some conformal changes of  K\"{a}hler Ricci solitons that do not preserve the GQE structure are  obtained in \cite{Maschler}. \end{remark}

The paper is organized as follows.  In section \ref{sec_warped_prod} we discuss warped product metrics with a one dimensional base.  The first observation, known to Brinkmann, is a characterization of these spaces as those admitting a non-constant solution to a certain PDE. The second observation is a duo of completeness lemmas for metrics conformal to a warped product, in which the conformal factor
is either a function on only the base or only the fiber.  Section \ref{sec_conf_changes} is the technical heart
of the paper.  We recast the GQE condition on $g$ in terms of an equivalent condition on a conformally rescaled metric $h$,
then establish a warped product structure on $h$.  Next, we understand the geometry of $g$ by demonstrating that
the conformal factor only depends on the fiber or the base, leading to two possible cases.  Finally, we prove
the global structure of $g$, arguing that both cases may not occur on a connected manifold.  In section \ref{sec_examples}
we give a variety of examples that demonstrate the sharpness of the classification theorems.  Section \ref{sec_conf_fields}
takes up the case in which a GQE manifold admits a structure-preserving conformal field, and section \ref{sec_solitons}
specializes our results to gradient Ricci solitons and $m$-quasi-Einstein manifolds.

\section{Warped Products over a one-dimensional base}
\label{sec_warped_prod}
In this preliminary section we recall the notion of warped products over a one-dimensional base and 
their characterization as the metrics that support a gradient conformal field.  This result has a long history: 
the local version goes back to Brinkmann \cite{Brinkmann}, and the global version was established in full generality in 
the  Riemannian case by Tashiro \cite{Tashiro}.  Tashiro's work  generalized a well-known characterization of the 
sphere due to Obata \cite{obata}.    

We require slightly  non-standard versions of these results where our metric is not complete, but is conformal to a 
complete metric by a conformal change of a certain form.  In Lemmas \ref{lemma_f(t)_global}
and \ref{lemma_f(x)_global}, we establish that Tashiro's proof can be 
extended to give a global warped product structure in these cases, a necessary step in our eventual
proof of Theorem \ref{thm_local_classification}.

\begin{definition}
\label{def_wp}
A \emph{warped product over a one-dimensional base} is a smooth manifold isometric to one of the following:
\begin{enumerate}
\item[(I)]  \[  \left( I \times N , h = dt^2 +  v(t)^2 g_N\right) \]
where $I$ is an open interval (possibly infinite), $v: I \to \R$ is smooth and positive, and $(N,g_N)$ is a Riemannian manifold,
\item [(II)]
\[ \left( B_R(0) \subset \mathbb{R}^n, h= dt^2 + v(t)^2 g_{S^{n-1}}\right), \]
where $B_R(0)$ is an open ball about the origin of radius $R \in (0,\infty]$,
$v: [0,R) \to \R$ is smooth, positive for $t>0$, with $v(0)=0$, and $g_{S^{n-1}}$ is a
round spherical metric, or
\item [(III)]
\[ \left(S^n, dt^2 + v(t)^2 g_{S^{n-1}}\right)\]
where $v: [0,R] \to \R$ is smooth, positive for $0 < t < R$, with $v(0)=v(R)=0$.
\end{enumerate}
\end{definition}

\begin{remark} In case (I) $g$ is complete if and only if $g_N$ is complete and $I=\R$.  Case (II) metrics
are complete if and only if $I= [0,\infty)$, and are rotationally-symmetric metrics on $\R^n$.  
Case (III) metrics are rotationally-symmetric metrics on $S^n$.
In cases (II) and (III) smoothness of the metric implies further boundary conditions 
on the derivatives of $v$ (see \cite{Petersen} for example).
\end{remark}

An important property of these spaces is they always support a gradient conformal vector field.
\begin{proposition}
\label{lemma_warped_prod_hess}
For an $n$-dimensional warped product metric over a one-dimensional base:
\[ h = dt^2 + v(t)^2 g_N, \]
any anti-derivative $u(t)$ of $v(t)$ satisfies the equation
\begin{equation}
\label{eqn_hessian} 
\frac{1}{2} L_{\nabla u} h = \mathrm{Hess}\, u = \frac{\Delta u}{n} h,
\end{equation}
where $\nabla$, $\mathrm{Hess}$, and $\Delta$ are the gradient, Hessian, and Laplacian with respect to $h$.
\end{proposition}
The fundamental fact we will exploit  is that the converse is also true.  For the full proof and history of this result we refer the reader to Lemma 3.6  and Proposition 3.8 of  
\cite{KR09}, which also include pseudo-Riemannian versions and additional references. 

\begin{lemma}
\label{lemma_hessian}
Suppose that a non-constant function $u$ on a Riemannian manifold $(M,h)$ satisfies  (\ref{eqn_hessian}). 
Then the  critical points of $u$ are non-degenerate and isolated.  Fix $p \in M$.
\begin{enumerate}
\item If $|\nabla u(p)| \neq 0$, then in a neighborhood $U$ of $p$, $g$ is isometric to a 
warped product over a one-dimensional base of type (I): 
\[ (U, g) \cong \left( I \times N , dt^2 + u'(t)^2 g_N \right), \]
where $u=u(t)$, $u'(t) \neq 0$ and $(N, g_N)$ is some Riemannian $(n-1)$-manifold independent of $t$.  If $x$ denotes
coordinates on $N$, we say $(t,x)$ give \emph{rectangular} coordinates on $U$.
\item If $|\nabla u(p)|=0$ then there is a neighborhood $U$ of $p$ on which $g$ is isometric
to a warped product over a one-dimensional base of type (II), and $u$ is a function of only
the distance $t$ to $p$:
\[ (U, g) \cong \left(B_R(0), dt^2 + \left( \frac{u'(t)}{u''(0)}\right)^2 g_{S^{n-1}} \right), \]
where $u'(t) \neq 0$ for $t>0$.  If $x$ denotes coordinates on $S^{n-1}$, we say $(t,x)$ give \emph{polar} coordinates on $U$.
\end{enumerate}
\end{lemma}

Later, in case (2) we rescale $g_{S^{n-1}}$ without further comment to absorb $u''(0)$.   Since aspects of the proof will 
be important for our next two results, we include a proof of  (1) for completeness. We will find the following definitions of Tashiro  useful.

\begin{definition} Let $u$ be a solution to (\ref{eqn_hessian}).  A 
\emph{$u$-component} is a connected component of a non-degenerate level set of $u$.  A \emph{$u$-geodesic} is a geodesic 
of $h$ that is parallel to $\nabla u$ wherever $\nabla u \neq 0$.  \end{definition}

\begin{proof}[Proof of Lemma \ref{lemma_hessian}, (1)]  Let $p$ be a point with  $\nabla u(p) \neq 0$ and let $L$ be the 
$u$-component containing $p$.   There is a  neighborhood $U$  of  $p$   such that  $\nabla u \neq 0$  on $U$, $U$ is 
diffeomorphic to $(-\varepsilon,\varepsilon) \times N$,  where $N \subset L$,  and $U$ has coordinates $(t,x)$, 
where $\frac{\partial}{\partial t} = \frac{\nabla u}{|\nabla u|}$ and  $x$ 
denotes coordinates for $N$.  We also choose $N$ to be connected.  

For $X$ orthogonal to $\nabla u$, (\ref{eqn_hessian}) implies that 
\[ D_{X} |\nabla u|^2 = 2 \mathrm{Hess } \,u(X, \nabla u) = 0, \]
so $|\nabla u|$ is a function of $t$ and  $\nabla u = \psi(t) \frac{\partial}{\partial t}$.  This in turn implies that 
$u=u(t)$ and $\psi(t) = u'(t)$. 
Then we have 
\[ h\left( \nabla_{\frac{\partial}{\partial t}} \frac{\partial}{\partial t}, X\right) = \frac{1}{u'(t)} \mathrm{Hess}\, u \left( \frac{\partial}{\partial t}, X \right) = 0, \]
which shows that the curves $t \mapsto (t,x)$ are the $u$-geodesics in $U$.

 This establishes that the metric is of the form $h= dt^2 + g_t$, where $g_t$ is a one-parameter family of metrics on $N$.  It also implies  that 
 \[ \mathrm{Hess}\,u \left(\frac{\partial}{\partial t}, \frac{\partial}{\partial t} \right) = h\left( \nabla_{\frac{\partial}{\partial t}} \left(u'(t) \frac{\partial}{\partial t}\right) , \frac{\partial}{\partial t}\right) = u''(t). \]
By (\ref{eqn_hessian}), we now see that $\Delta u = n u''(t)$. Using  (\ref{eqn_hessian})   again we obtain for 
$X, Y$ orthogonal to $\nabla u$, 
 \[ (L_{\frac{\partial}{\partial t}} h)(X,Y) = 2 \frac{u''}{u'} h(X,Y),\]
 which implies that $g_t(X,Y) = u'(t)^2 g_N(X,Y)$ for some fixed metric $g_N$ on $N$.
 \end{proof}

\begin{remark} The fact that  $|\nabla u|$ is a function of $t$ in the proof shows that we can choose $U$ to be a 
neighborhood of $L$, even when $L$ is non-compact: if $\nabla u \neq 0$ on $\{t_0\} \times N$, 
then $\nabla u \neq 0$ and has constant length on the whole leaf $\{t_0\} \times L$. In particular, we see that,  in a neighborhood 
of a point with $\nabla u(p) \neq 0$, the sets 
$\{ t\} \times N$ are the $u$-components and that the $u$-geodesics are the geodesics in the $t$ direction.    
In the case where $\nabla u(p) = 0$ we have that all of the geodesics beginning at $p$ are $u$-geodesics and the 
metric spheres around $p$ are $u$-components.
\end{remark} 

Tashiro's theorem is  the following  global version of this result: if $h$ is complete and supports a 
non-constant function satisfying   (\ref{eqn_hessian}), then $h$ is globally a warped product over a one-dimensional 
base (cf. Theorem 5.4 of \cite{OsgoodStowe}).  We show that 
Tashiro's arguments can be used to also prove global theorems for (possibly incomplete) metrics that are conformal 
to a complete metric in a certain way. First we need a definition.   

\begin{definition}
Let $u$ be a non-constant solution to (\ref{eqn_hessian}) on $(M,h)$, let $f$ be a smooth function on $M$, and let $U \subset M$.
We say $f=f(t)$ on $U$ if $\nabla f$ is parallel to $\nabla u$ on $U$, and $f=f(x)$ on $U$ if $\nabla f$
is orthogonal to $\nabla u$ on $U$.
 \end{definition}
\begin{remark} 
From Lemma \ref{lemma_hessian}, around every point $p$, the metric $h$ can be written on some neighborhood $U$ of $p$ as
\[ dt^2 + u'(t)^2 g_N \]
with (polar or rectangular) coordinates $(t,x)$. Then $f=f(t)$ in the above definition if and only if $f$ is a function of only the $t$-coordinate on $U$;
similarly, $f=f(x)$ as above if and only if $f$ is independent of $t$ on $U$.
\end{remark}
 
\begin{lemma} \label{lemma_f(t)_global}  Suppose that a non-constant function $u$ on a Riemannian manifold $(M,h)$ satisfies (\ref{eqn_hessian}) and suppose that $f$ is a function such that  $f=f(t)$ on $M$ and  $(M, e^{\frac{2f}{n-2}} h)$ is complete.  Then $(M,h)$ is (globally) a one-dimensional warped product, with fiber metric $g_N$ complete.
\end{lemma}

\begin{remark} The normalization of the conformal factor $e^{\frac{2f}{n-2}}$ is not important, but is used to be consistent with later notation. \end{remark} 
\begin{proof}
Set $g=e^{\frac{2f}{n-2}}h$.  Let $N$ be a $u$-component, with induced metric $g_N$ from $h$.  Applying Lemma \ref{lemma_hessian}
 to every  point of $N$ shows that $g|_{TN}$ and ${h}|_{TN}$ are homothetic (since $f=f(t)$).  Since $N$ is a closed subset of the complete 
manifold $(M,g)$, it follows that $(N, g_N)$ is complete. 

Let $J$ be the largest open interval of regular values of $u$ that contains $u(N)$, and let $U \subset M$ be the connected 
component of $u^{-1}(J)$ that contains $N$.  Let $q \in N$ and let $\gamma_{q}$ be the $u$-geodesic with respect to $h$ 
through $q$.  Since $f=f(t)$, $\gamma_{q}$ is, 
up to reparametrization, also a geodesic for $g$.  In particular, since $g$ is complete, such curves are well-defined  until they possibly leave $U$.  Moreover,  since $u=u(t)$,
they all leave $U$ (if at all) at the same parameter value of $t$.

As in the proof of Lemma \ref{lemma_hessian}, it follows that $U$ is diffeomorphic  to $I \times N$, where $I$ 
is an open interval, and that in the coordinates induced by this diffeomorphism,
\[h = dt^2 + u'(t)^2 g_N,\]
where $t$ is the signed $h$-distance to $g_N$.  Say $I=(a,b)$ (where $a,b \in [-\infty,\infty])$, and
define the change of variables:
\[s(t) = \int_0^t e^{\frac{f(r)}{n-2}} dr.\]
Using the new coordinate $s$, we have that on $U$,
\[g = ds^2 + u'(s)^2 e^{\frac{4f(s)}{n-2}} g_N.\]
Define
\[ c = \lim_{t \to a+} s(t), \qquad \qquad d = \lim_{t \to b-} s(t).\]
Now we can just imitate Tashiro's proof, analyzing three possible cases.

If $c=-\infty$ and $d= +\infty$, the restriction of $g$ to the open subset $U$ defines a complete metric, 
and so $(U,g|_U) = (M,g)$. In particular, $g$ is a globally warped product with one-dimensional base of type (I).
Since $f=f(t)$, the same goes for $h$.

Next, suppose $c=-\infty$ but $d$ is finite (or vice versa).  Consider a geodesic $\gamma(s)$ with respect to $g$, orthogonal to 
$N$ with increasing $s$. By completeness, $\gamma$ may be extended to $\R$, and so we conclude $q=\lim_{s \to d-} \gamma(s)$
is a critical point of $u$ (or else $J$ was not maximal as chosen). By Lemma \ref{lemma_hessian}, $h$ admits
polar coordinates $(t_1,x)$ about $p$ with warping factor $u'$ and fiber $S^{n-1}$.  Since the coordinates $t$ 
and $t_1$ are both given by level sets of $u$, these coordinate neighborhoods can be combined to show that $h$
and $g$ are warped product metrics with one-dimensional base of type (II).

Finally, consider the case in which $c$ and $d$ are both finite.  A similar argument shows that $u$ has critical points
at $s=c$ and $s=d$, and we conclude that $h$ is a warped product with one-dimensional base of type (III).
\end{proof}
From these arguments we also see the following in the case $f=f(x)$.
\begin{lemma}  
\label{lemma_f(x)_global}
Suppose that a non-constant function $u$ on a Riemannian manifold $(M,h)$ satisfies 
(\ref{eqn_hessian}) and has no critical points.  Suppose that $f$ is a function such that  
$f=f(x)$ on $M$ and  $(M, e^{\frac{2f}{n-2}} h)$ is complete.  Then $(M,h)$ is (globally) a one-dimensional warped product 
of type (I):
\[ (M,h) = \left( \mathbb{R} \times N, dt^2 + u'(t)^2 g_N\right). \]
\end{lemma}

\begin{remark} In this case, $g_N$ is not necessarily complete. \end{remark}

\begin{proof}
This result follows from similar arguments to Lemmas \ref{lemma_hessian} and \ref{lemma_f(t)_global} once we show that the $u$-geodesics with respect to $h$ exist for all time.   Let $\gamma(t)$ be a $u$-geodesic with respect to $h$ defined for $0 \leq t < t_0$.   Let $t_i \nearrow t_0$, so that $\{\gamma(t_i)\}$ is Cauchy in $(M,h)$.  The sequence  is also Cauchy in  $(M,g)$
since $g=e^{\frac{2f}{n-2}}h$ and $f$ is constant along $\gamma(t)$.
By completeness, $\{\gamma(t_i)\}$ converges with respect to $g$ to some $q \in M$.  By considering
the $u$-geodesics  in a neighborhood of $q$, we see that $\gamma$ can be extended past $t_0$.
\end{proof}

\begin{remark} From the proof we can see that only the completeness of the $u$-geodesics in $g$ is necessary for this last result. \end{remark}

\section{Conformal diffeomorphisms preserving the GQE structure}
\label{sec_conf_changes}
In this section we prove Theorem \ref{thm_local_classification}, giving the classification of GQE metrics 
admitting a conformal diffeomorphism preserving the GQE structure.  The first step is to give a convenient conformal 
interpretation of the GQE equation (\ref{eqn_GQE}), used previously by Catino \cite{Catino} and Kotschwar \cite{Kotschwar10}.

\begin{proposition} A Riemannian manifold $(M,g)$ of dimension $n \geq 3$ satisfies the GQE equation (\ref{eqn_GQE}) 
with functions $f, \alpha,$ and $\lambda$ if and only if there is a conformally related metric $h$ that satisfies 
\begin{equation}
\label{eqn_Ric_h}
\mathrm{Ric}_{h} = \left(\frac{1}{n-2} -\alpha\right) df \otimes df + Q  h
\end{equation}
for some function $Q$, where $\mathrm{Ric}_h$ is the Ricci curvature of $h$.
\label{prop_Ric_h}
\end{proposition}

\begin{proof} 
Set  $h=e^{\frac{-2f}{n-2}} g$.  The Ricci curvatures of $h$ and $g$ are related by:
\begin{equation}
\label{eqn_Ric_conformal}
\mathrm{Ric}_{h} = \mathrm{Ric}_g + \mathrm{Hess}_g f + \frac{1}{n-2}df \otimes df + \frac{1}{n-2}\left( \Delta_g f - |\nabla f|_g^2\right) g.
\end{equation}
Thus we see that $h$ satisfies (\ref{eqn_Ric_h}) if and only if $\mathrm{Ric}_g + \mathrm{Hess}_g f + \alpha df \otimes df = \lambda g$,  where 
\begin{equation} \label{Qeqn}
 Q = \frac{1}{n-2}\left( \Delta_g f - |\nabla f|_g^2 + (n-2)\lambda\right)e^{\frac{2f}{n-2}}. 
 \end{equation}
\end{proof}

\begin{remark} $h$ is generally incomplete even if $g$ is complete.\end{remark}  
\begin{remark} It follows that a Riemannian metric is conformal to an Einstein metric if and only if it admits a 
GQE structure with $\alpha \equiv \frac{1}{n-2}$.
\end{remark}  

\begin{example}
\label{example_f(t)}
A warped product over a one-dimensional base
\[h=dt^2 + v(t)^2 g_N\]
has Ricci curvature 
\begin{equation}
\label{eqn_Ric_h_WP}
\mathrm{Ric}_h = -(n-1)\frac{v''}{v} dt^2 + \mathrm{Ric}_{g_N} - \left(vv''+(n-2)(v')^2\right)g_N.
\end{equation}
Assume $\mathrm{Ric}_{g_N} = \mu g_N$ for a constant $\mu$.  
Then (\ref{eqn_Ric_h})
is satisfied with $f=f(t)$ if and only if 
\begin{eqnarray}
\label{eqn_f_prime_Ric}
\left( \frac{1}{n-2} - \alpha \right) f'(t)^2 & =&  \mathrm{Ric}_h\left( \frac{\partial}{\partial t},\frac{\partial}{\partial t} \right) -  \mathrm{Ric}_h\left( X,X \right)  \\
  &=&  \frac{(n-2) ((v')^2-vv'')-\mu}{v^2} \nonumber
\end{eqnarray}
where $X$ is a field perpendicular to $\frac{\partial}{\partial t}$  such that $h(X,X) = 1$.  Choose a function $\alpha(t)$
so that (\ref{eqn_f_prime_Ric}) defines a function $f(t)$.  Then $h$ and $f$ satisfy (\ref{eqn_Ric_h}) for some $Q$.  
Consequently, $g=e^{\frac{2f}{n-2}}h$ is a GQE metric.
\end{example}

\begin{example}
As an alternative to the above, let $(M,g)$ be a warped product over a one-dimensional base,
$g= ds^2 + v(s)^2 g_N$, of type (I) with
$g_N$ Einstein (with Einstein constant $\mu$), or else of type (II) or (III) (in which the Einstein constant
of $g_{S^{n-1}}$ is $(n-1)$.  Then $(M,g)$ is automatically a Ricci almost soliton, where the potential $f=f(s)$
can be found from the ODE:
\[\left(\frac{f'}{v}\right)' = \frac{\mu + (n-2)(vv''-(v')^2)}{v^3},\]
and $\lambda=\lambda(s)$ is given by:
\[\lambda = f'' - (n-1)\frac{v''}{v}.\]
\end{example}

\subsection{Local form of $h$} 
We first prove a local classification for the conformally rescaled metric $h$.
\begin{lemma}
\label{lemma_local_form_h}
Let $(M_1,g_1,f_1,\alpha_1,\lambda_1)$ and $(M_2,g_2,f_2,\alpha_2,\lambda_2)$ be GQE manifolds 
admitting a non-homothetic conformal diffeomorphism $\phi$ preserving the GQE structure. 
Then every $p \in M_1$ is contained in a neighborhood $U$ on which $h_1=e^{-\frac{2f_1}{n-2}} g_1$ 
is a warped product over a one-dimensional base of type (I) or (II):
\[h_1 = dt^2 + u'(t)^2 g_N\]
for an appropriate function $u(t)$.
\end{lemma}

\begin{remark}  If $\alpha_1=\alpha_2=\frac{1}{n-2}$ or $f_1$ and $f_2$ are constant, this result recovers Brinkmann's original 
result for Einstein manifolds.
\end{remark}

\begin{proof} 
Let $h_i=e^{-\frac{2f_i}{n-2}} g_i$ be the corresponding conformally rescaled metrics.  Then $\phi^* g_2 = w^{-2} g_1$
if and only if $\phi^*(h_2) = u^{-2} h_1$,
where 
\[u^{-2}=w^{-2} e^{\frac{2f_1}{n-2}}e^{-\frac{2\phi^*f_2}{n-2}} = w^{-2}e^{-\frac{2C}{n-2}}\]
for the constant $C=\phi^*f_2 - f_1$.  In particular, since
$w$ is non-constant, $u$ is non-constant as well.
We have from Proposition \ref{prop_Ric_h} that
$\mathrm{Ric}_{h_i} = \left(\frac{1}{n-2} - \alpha_i\right)df_i \otimes df_i + Q_i h_i $ for $i=1,2$.  
Since $\phi$ preserves the GQE structure, we have
\[ \phi^*\left ( \left( \frac{1}{n-2} - \alpha_2 \right)  df_2 \otimes df_2 \right) =\left( \frac{1}{n-2} - \alpha_1\right) df_1 \otimes df_1,\]
and so
\[ \mathrm{Ric}_{\phi^*( h_2)} - \mathrm{Ric}_{h_1} =  (\phi^*Q_2) \phi^*h_2 - Q_1 h_1 = \left((\phi^*Q_2) u^{-2} - Q_1\right) h_1.\]
Thus the difference of the Ricci tensors is pointwise proportional to $h_1$. 
On the other hand,  by the formula for the conformal change $h_1 \to u^{-2}h_1$, we have
\[ \mathrm{Ric}_{\phi^* (h_2)} - \mathrm{Ric}_{h_1} =  (n-2)  \frac{\mathrm{Hess}_{h_1} u }{u}  + \left( u^{-1}\Delta_{h_1} u  - (n-1) |\nabla u|_{h_1}^2 u^{-2} \right) h_1.\]
Putting these equations together we conclude that $\mathrm{Hess}_{h_1} u$ is pointwise proportional to $h_1$.  Taking
the trace, we see that (\ref{eqn_hessian}) is satisfied by $u$ with respect to $h_1$.  From Lemma
\ref{lemma_hessian}, we deduce the local warped product structure of $h_1$.
\end{proof}

\subsection{Local form of $g$} Now with a local classification of the metric $h_1=e^{-\frac{2f_1}{n-2}}g_1$, we 
pass to a local classification of $g_1$ by understanding the local behavior of $f_1$.

\begin{lemma} Under the hypotheses of Lemma \ref{lemma_local_form_h}, let $p$ be a point in $M_1$ with 
$\alpha_1(p) \neq \frac{1}{n-2}$, and set $h=h_1$ and $f=f_1$, and $\alpha=\alpha_1$.
\label{lemma_local_form_f}
\begin{enumerate}
\item If $p$ is a critical point of $u$, then $h$ is a type (II) warped product
\[ h= dt^2 +  u'(t)^2 g_{S^{n-1}} \]
in a polar coordinate neighborhood $U$ of $p$ and $f=f(t)$ on $U$. 
\item If $p$ is not a critical point of $u$, then $h$ is a type (I) warped product
\[ h = dt^2 +  u'(t)^2 g_{N}  \]
 in a rectangular coordinate neighborhood $U$ of $p$, and either 
\begin{enumerate}
\item $f=f(t)$ on $U$ and (if $n \geq 4$ or $\alpha$ is constant) $g_N$ is Einstein, or 
\item $f=f(x)$ on $U$ and 
\begin{equation}
\label{eqn_Ric_gN}
 \mathrm{Ric}_{g_N} = \left(\frac{1}{n-2}-\alpha\right)df \otimes df + P g_N
\end{equation}
where $P$ is a constant and $\alpha=\alpha(x)$.    Moreover $Q$, defined in Proposition \ref{prop_Ric_h},
is constant and $u'''= \frac{-Q}{n-1} u'$.  
\end{enumerate}
\end{enumerate}
Finally, case \emph{(2b)} does not occur if $n=3$.
\end{lemma}
In particular, in the $f=f(x)$ case, the metric $e^{\frac{2f}{n-3}} g_N$ is a GQE $(n-1)$-manifold with potential $f$, 
with $\alpha$ shifted by $\frac{1}{n-3}- \frac{1}{n-2}$ (by Proposition \ref{prop_Ric_h}).

\begin{proof}
By the previous lemma, in a neighborhood $U$ of $p$,  
\[ h = dt^2 + u'(t)^2 g_N, \]
where $(t,x)$ are either polar or rectangular coordinates and the metric $g_N$ is independent of $t$.  
Letting $v(t) = u'(t)$, we find the Ricci curvature of $h$ from (\ref{eqn_Ric_h_WP}). For $X, Y$ tangent to $N$,
\begin{eqnarray} \label{RicCalcWP}
\mathrm{Ric}_{h}\left( \frac{\partial}{\partial t}, \frac{\partial}{\partial t}\right) &=& - (n-1)\frac{v''}{v}\\
\mathrm{Ric}_{h}\left( \frac{\partial}{\partial t}, X\right)&=& 0 \nonumber\\
\mathrm{Ric}_{h}(X, Y)&=&  \mathrm{Ric}_{g_N}(X,Y) - \left( \frac{v''}{v} +(n-2) \frac{(v')^2}{v^2} \right) h(X,Y) \nonumber
\end{eqnarray}
In this proof, we frequently identify $\{t\} \times N$ with $N$.

We begin with some observations.  First, from (\ref{eqn_Ric_h}) we see that $\mathrm{Ric}_h$ has at most two distinct eigenvalues at each point. If
there are two distinct eigenvalues, the orthogonal eigenspaces are of dimension $1$ and $n-1$.  Second, 
  $\nabla f$ is an eigenvector field for the $1$-dimensional eigenspace of $\mathrm{Ric}_h$ wherever it does not vanish.
Third, from (\ref{RicCalcWP}), $\frac{\partial}{\partial t}$ is an eigenvector field for $\mathrm{Ric}_h$.

Fix $p \in M$, and let $U$ be a coordinate neighborhood as in Lemma \ref{lemma_local_form_h} (shrunken if necessary so that 
$\alpha \neq \frac{1}{n-2}$ on $U$).
\paragraph{\emph{Case A}}
If $p$ is a critical point of $u$ then we have polar coordinates $(t,x)$, and $g_N=g_{S^{n-1}}$.  We then have
\begin{eqnarray*}
\mathrm{Ric}_{h}(X, Y)&=& \left( \frac{\mu - (n-2)(v')^2}{v^2} - \frac{v''}{v}  \right) h(X,Y).
\end{eqnarray*}
where $\mu$ is the Einstein constant of $g_{S^n}$.   This shows that the $n-1$ dimensional space  $T_qN \subset T_qM$  is contained in an eigenspace of $\mathrm{Ric}_h$ for every $q \in U$.    Therefore,  at any point in 
$U$ for which $\nabla f \neq 0$, we have that $\nabla f$ and $\frac{\partial}{\partial t}$ both span the one-dimensional eigenspace of $\mathrm{Ric}_{h}$ so that $\nabla f$ is parallel to $\frac{\partial}{\partial t}$.  Then $f=f(t)$ on $U$.  

From now on, we assume $p$ is not a critical
point of $u$, so that we have rectangular coordinates $(t,x)$ on $U$.  Without loss of generality we also assume 
that $p=(0,x_0)$ in these coordinates.   

\paragraph{\emph{Case B}}
Suppose $p$ is not a critical point of $u$ and $\mathrm{Ric}_{g_N} = \mu g_N$ at all points in a neighborhood $V\subset N$ 
containing  $x_0$ for some function $\mu \in C^{\infty}(V)$.   (This assumption is always satisfied when $n=3$ and by Schur's 
lemma $\mu$ is constant if we are in this case and $n>3$).   Then the exact same argument as in Case A, which only used   $\mathrm{Ric}_{g_N} = \mu g_N$, shows that $f=f(t)$ on $U$. 
Moreover, if $\alpha$ is constant, then (\ref{eqn_Ric_h}) and (\ref{RicCalcWP}) show that $Q=Q(t)$ and consequently that
$\mu$ is independent of $x$ and therefore constant.

We are now left with the case that $p$ is not a critical point of $u$ and  $\mathrm{Ric}_{g_N} \neq \mu g_N$ in any neighborhood of $x_0$.  There are then two cases, depending on whether the condition is true at $x_0$ or not.  

\paragraph{\emph{Case C}} 
Suppose $p$ is not a critical point of $u$  and $\mathrm{Ric}_{g_N}$ is not proportional to $g_N$ at $x_0$. 

Then $\mathrm{Ric}_h$ is not proportional to $h$, so
$\nabla f(p) \neq 0$.  Shrink $U$ if necessary so that $\nabla f \neq 0$ on $U$.  $\nabla f$ is
an eigenvector field of $\mathrm{Ric}_h$ with corresponding eigenvalue of multiplicity one.  It cannot be parallel to $\frac{\partial}{\partial t}$ and therefore
must be orthogonal to $\frac{\partial}{\partial t}$ on $U$.
This shows that $f=f(x)$ on $U$.  Setting our two expressions for $\mathrm{Ric}_h$ equal yields:
\begin{align}
-(n-1)\frac{v''}{v} &= Q\nonumber\\
\mathrm{Ric}_{g_N} &= \left(\frac{1}{n-2}-\alpha\right) df \otimes df + \underbrace{\left(Qv^2 + vv'' + (n-2)(v')^2\right)}_{P}g_N \label{eqn_g_N_P}
\end{align}
The first equation implies that $Q=Q(t)$; it follows then from the second that the coefficient $P$
on $g_N$ is constant, and consequently $\alpha=\alpha(x)$. Eliminating $Q$ in the equations gives
\[  v v'' - (v')^2  = - \frac{P}{n-2}. \]
Any solution to this equation must solve $v'' = kv$ for some constant $k$ and therefore $Q$ must also be constant.  We then
find for later reference that
\begin{equation}
\label{eqn_vPQ}
(v')^2 + \frac{Q}{n-1}v^2 = \frac{P}{n-2}.
\end{equation}

\paragraph{\emph{Case D}} 
$p$ is not a critical point of $u$ and $\mathrm{Ric}_{g_N}$ is proportional to $g_N$ at $x_0$, but is not proportional in 
any neighborhood of $x_0$. 

Consider a sequence of  points $x_i$ in $N$ converging to $x_0$ such that $\mathrm{Ric}_{g_N}$ is not proportional to 
$g_N$ at $x_i$.  Then $\nabla f(t,x_i)$ must be tangent to $N$ by  case C.  Therefore, by continuity and (\ref{eqn_g_N_P}),  we must have $df = 0$ at the points $(t, x_0)$ $t\in I$.    If $df=0$ in a neighborhood of $p$, then $f$ is constant and the lemma is clearly true by simply  shrinking $U$ to be the neighborhood where $f$ is constant.    

Consider a connected component $W$ of the nonempty, open set $\{df\neq0\} \cap U$ with $p \in \partial W$.
Since   $df\neq0$ on $W$, we have by cases A--C that either $f=f(t)$ or $f=f(x)$ on $W$.  
By way of contradiction, suppose $f=f(t)$ on $W$. Then $W$ is a set of the form $(a,b) \times V$, where $V$ is an open subset of $N$.
By the previous paragraph, $f$ is constant along the curve $ t \mapsto (t, x_0)$.  If $x_0 \in V$,
this shows that $df$ vanishes in $W$, a contradiction.  If $x_0 \in \partial V$, the same
argument applies by continuity.
Therefore, we have $f=f(x)$ in a neighborhood of $p$, so we may follow the argument of case C. 
\end{proof}

\subsection{Global form of $g$}
\label{sec_conf_changes_global}
The previous lemma splits $M$ into two sets: the points where $f=f(t)$ and the points  where $f=f(x)$.  We now
rule out the possibility that both cases occur.

\begin{lemma}
\label{lemma_ft_fx}
If $\alpha \neq \frac{1}{n-2}$, 
the cases $f=f(x)$ and $f=f(t)$ may not both occur on the same connected manifold $M$, unless $f$ is constant.  Moreover, if
$f$ is non-constant and $f=f(x)$ occurs, then $u$ has no critical points.
\end{lemma}
\begin{proof}
Define $A$ to be the set of points $p \in M$ that are either critical points of $u$ or 
regular points $p$ that satisfy the property
\begin{quotation}
$\nabla f$ is everywhere orthogonal to $L$ and
$|\nabla f|$ is constant along $L$, where $L$ is the $u$-component containing $p$,
\end{quotation}
which we denote by $(*)$.  As usual, $\nabla$ is the gradient with respect to $h$.

We show $A$ is open.  Let $p \in A$.  First, if $p$ is a critical point, then by the previous lemma, we have polar
coordinates around $p$ with $f=f(t)$.  On this coordinate neighborhood, $(*)$ clearly holds at every point besides $p$.  Otherwise,
$p$ is a regular point; let $L$ be the $u$-component containing it.  If $\nabla f(p) \neq 0$, then we are in the $f=f(t)$ case on
a neighborhood $U$ constructed in the previous lemma.  It is readily seen that $(*)$ holds on $U$.

If $\nabla f(p) = 0$, then by $(*)$, $\nabla f$ vanishes on $L$.  Then
on $L$, the Ricci curvature of $L$ is proportional to the metric on $L$, by (\ref{eqn_Ric_h}) and (\ref{eqn_Ric_h_WP}).
Let $U$ be a coordinate neighborhood of $p$ as in the previous lemma, so that
\[h=dt^2 +u'(t)^2 g_N\]
on $U$, where $N \subset L$. By case B of the proof of the previous lemma, $f=f(t)$ on $U$, 
and so $(*)$ holds on $U$.

Next, we show $A$ is closed.  Let $\{p_i\}$ be a sequence in $A$ converging to $p \in M$.
Since $A$ is open and the critical points of $u$ are isolated, we may assume without loss of generality
that each $p_i$ is a regular point of $u$.  If $p$ is a critical point of $u$, we are done.  Otherwise,
let $L$ be the $u$-component containing $p$, and similarly $L_i$ for $p_i$.  It is now clear from the definition that $(*)$ 
holds on $L$, since it holds on each $L_i$.

Thus, either $A$ is empty, or $A=M$.  If $f=f(x)$ and is non-constant on some open set, then $\nabla f$ is
tangent to a $u$-component, and so $A$ is empty.  In particular, $u$ has no critical points.  If, in addition,
$f=f(t)$ and is non-constant on an open set, then $A$ is non-empty, a contradiction.
\end{proof}

Now we complete the proof of the local and global classifications stated in the introduction.
\begin{proof}[Proof of Theorem \ref{thm_local_classification}]
From Lemmas \ref{lemma_local_form_h} and \ref{lemma_local_form_f}, $g_1=e^{\frac{2f_1}{n-2}} h_1$ is locally either of the form
(\ref{eqn_g_f(t)}) or  (\ref{eqn_g_f(x)}), (since in the $f_1=f_1(t)$ case we may do a change of variables
$ds = e^{\frac{f_1(t)}{n-2}} dt$).  Moreover, if $g_1$ is complete, then we have a global structure of the form 
(\ref{eqn_g_f(t)}) or  (\ref{eqn_g_f(x)}) by Lemmas \ref{lemma_f(t)_global} and \ref{lemma_f(x)_global}.
In the $f_1=f_1(x)$ case, (\ref{eqn_Ric_gN}) is satisfied, which implies by
Proposition \ref{prop_Ric_h} that $e^{\frac{2f_1}{n-3}} g_N$ is GQE with potential $f_1$.

Next, we prove that $h_2$ also has a warped product structure.  Since
\begin{eqnarray*}
h_1 = dt^2 + u'(t)^2 g_N,
\end{eqnarray*}
the metric $h_2$ satisfies 
\begin{eqnarray*}
\phi^{*}(h_2) = u^{-2}(t) dt^2 + (u' u^{-1})^2 g_N 
\end{eqnarray*}
Defining $dr = u^{-1}(t) dt$ we then obtain that
\[ h_2 = dr^2 + \left( \frac{d}{dr} (- u^{-1}) \right) ^2 g_N \]
up to isometry.  Note that since $\phi^* df_2 = df_1$, if $f_1= f_1(t)$ then $f_2= f_2(t)$ 
(and similarly if $f_1=f_1(x)$). Thus, $g_2$ has the form stated in the theorem.
\end{proof}

We also prove  the global result for compact manifolds.

\begin{proof}[Proof of Theorem \ref{thm_compact_case}]
By compactness and Lemma \ref{lemma_f(t)_global} (or alternatively by Tashiro's theorem), $h_1$ is a warped product
with one-dimensional base of type (III), and is therefore a rotationally-symmetric metric on a sphere.
Such metrics are conformal to a round metric, so certainly $(M_1, g_1)$ is conformally diffeomorphic to a round sphere, and the same goes from $(M_2, g_2)$.

Suppose $\alpha_1 \neq \frac{1}{n-2}$.  Since $u$ has a critical point by compactness,
we are in the case $f_1=f_1(t)$ by Lemma \ref{lemma_ft_fx}.   It then follows that $(M_1, g_1)$ and $(M_2, g_2)$ are rotationally-symmetric metrics on the sphere.
\end{proof}

\section{Examples}
\label{sec_examples} 
In the next two subsections we construct examples of generalized quasi-Einstein manifolds that are warped products 
over a one-dimensional base, both in the $f=f(t)$ and $f=f(x)$ cases.  By Proposition \ref{prop_Ric_h}, it suffices
to construct metrics $h$ and functions $f$ satisfying (\ref{eqn_Ric_h}).  In the third subsection we further show that all of these examples admit  one-parameter families of local conformal changes that preserve the GQE structure.  In many cases  these conformal changes are global. 

\subsection{The $f(t)$ case}
\label{sec_f(t)}
 We start with an arbitrary  Riemannian manifold $(U,h)$ of dimension $n \geq 3$ of the form 
\begin{eqnarray*}
U &=& (a,b) \times N \\
h &=& dt^2 + u'(t)^2 g_N 
\end{eqnarray*}
for some $u(t)$ 
with $u'(t)>0$ on $(a,b)$. Assume that $N$ is an Einstein metric, $\mathrm{Ric}_{g_N} = \mu g_N$ and that  $\alpha:(a,b) \rightarrow \mathbb{R}$ is also a smooth function of $t$, such that $\alpha(t) \neq \frac{1}{n-2}$ for any $t$.   
From example \ref{example_f(t)} we see that there is  metric of the form $g = e^{\frac{2f}{n-2}} h$  for some function $f=f(t)$   if $u$ satisfies the following differential inequality. 
\begin{equation}
\label{eqn_diff_ineq}
\frac{1}{ \frac{1}{n-2} - \alpha} \left(  -(n-2)  \frac{u'''}{u'} - \frac{\mu - (n-2) (u'')^2}{(u')^2}  \right) \geq 0 
\end{equation}
on $(a,b)$. 

\begin{remark} This shows that, given a warping function $u>0$ such that the derivatives of $u$ are bounded on  $(a,b)$, and  any $\alpha\neq \frac{1}{n-2}$, there is choice of Einstein metric $g_N$ so that the metric admits a GQE structure on $(a,b)$.  \end{remark}

From formula (\ref{eqn_f_prime_Ric}) an equivalent  way to state this result is as follows.
\begin{proposition}
Let  $(U,h)$, $(N,g_N)$ and $\alpha$ be as above.   Then  $(U,e^{\frac{2f}{n-2}} h, f, \alpha, \lambda)$ is a GQE manifold for some $\lambda$ and $f$ 
if and only if 
\begin{enumerate} 
\item $\mathrm{Ric}_h\left( \frac{\partial}{\partial t},\frac{\partial}{\partial t} \right)  \geq \mathrm{Ric}_h\left( X,X \right)$ on $(a,b)$, when $\alpha < \frac{1}{n-2}$, and 
\item $\mathrm{Ric}_h\left( \frac{\partial}{\partial t},\frac{\partial}{\partial t} \right)  \leq \mathrm{Ric}_h\left( X,X \right)$ on $(a,b)$, when $\alpha > \frac{1}{n-2}$.
\end{enumerate} 
\end{proposition}

\begin{example} We consider the concrete example 
\[ h = dt^2 + e^{2t} g_N \]
so that 
\[\left(\frac{1}{n-2}-\alpha\right)(f')^2 = - \mu e^{-2t}\]
by (\ref{eqn_f_prime_Ric}). 
If $g_N$ has Einstein constant $\mu < 0$, we 
may choose $\alpha = 0$, for instance, so that $f$ equals
\[ f(t) = \pm \sqrt{-\mu(n-2)} e^{-t} + C.\] 
If $\mu > 0$ we can choose $\alpha=\frac{2}{n-2}$, for instance, and so 
$f(t) = \pm \sqrt{\mu(n-2)} e^{-t} + C$. 
Finally, if $\mu=0$, then $h$ is Einstein.
\end{example}

This construction also works in polar coordinate neighborhoods. 
 \begin{proposition} 
 Suppose that 
 \begin{eqnarray*}
U &=& B_R(0) \\
h &=& dt^2 + u'(t)^2 g_{S^{n-1}}
\end{eqnarray*}
is a polar coordinate neighborhood of a smooth metric $h$,  and suppose that $\alpha=\alpha(t)$ is a smooth
function on $U$, never equal to $\frac{1}{n-2}$. Then $(U,e^{\frac{2f}{n-2}} h, f, \alpha, \lambda)$ is a GQE manifold for some $\lambda$ and $f$ 
if and only if (\ref{eqn_diff_ineq}) holds on $[0,R)$ with $\mu=n-2$.
\end{proposition}

\begin{proof}
Define $f(t)$ to solve (\ref{eqn_f_prime_Ric}) on $(0,R)$ (where $v=u'$).
We know that  $u' \rightarrow 0$ as $t \rightarrow 0$.  From the smoothness of the metric, it also follows that  
  \[ \mathrm{Ric}_h\left( \frac{\partial}{\partial t},\frac{\partial}{\partial t} \right) \longrightarrow \mathrm{Ric}_h\left( X,X \right) \quad   \text{as} \quad t \rightarrow 0. \]
Since $\alpha \nrightarrow \frac{1}{n-2}$ as $t \rightarrow 0$, $f$ extends to a smooth function on $U$ with a critical point 
at $t=0$.  
\end{proof}

If we do not prescribe $\alpha$, these propositions give the following corollary.  
\begin{corollary}
Let $(U,h)$ be any warped product over a one-dimensional base with fiber metric $g_N$ Einstein.   
If $h$  has non-constant curvature at almost every point, then there are functions $\alpha=\alpha(t)$, $f=f(t)$, and $\lambda=\lambda(t)$ such that $(U,e^{\frac{2f}{n-2}} h, f, \alpha, \lambda)$ is a \emph{complete} generalized quasi-Einstein structure.     \end{corollary}

\begin{proof}
First define $\alpha(t)$ so that it has the following properties:
\begin{itemize}
\item  $\alpha(t)  < \frac{1}{n-2}$  at points in $(a,b)$ where,  $\mathrm{Ric}_h\left( \frac{\partial}{\partial t},\frac{\partial}{\partial t} \right)  \geq \mathrm{Ric}_h\left( X,X \right)$, 
\item $\alpha(t)  > \frac{1}{n-2}$ at points in $(a,b)$ where,  $\mathrm{Ric}_h\left( \frac{\partial}{\partial t},\frac{\partial}{\partial t} \right)  \leq \mathrm{Ric}_h\left( X,X \right)$, and 
\item $\alpha(t)  = \frac{1}{n-2}$ at points in $(a,b)$ where,  $\mathrm{Ric}_h\left( \frac{\partial}{\partial t},\frac{\partial}{\partial t} \right)  = \mathrm{Ric}_h\left( X,X \right)$. 
 \end{itemize}
Defining $f(t)$ via (\ref{eqn_f_prime_Ric}), we have a GQE manifold structure on $(a,b) \times N$. By
the hypothesis on curvature, $f'$ vanishes only on a set of measure zero.
If $h$ is a type 
(I) warped product, this gives a GQE structure on $U$.   Then we can also  choose $\alpha(t) \rightarrow \frac{1}{n-2}$  fast enough as $t$ 
limits to $a$ and $b$ so that we can make $f'(t)$ blow up at the endpoints  so that $s(t) = \int_0^t e^{\frac{f}{n-2}} dt$ limits to $-\infty$ as $t \rightarrow a^+$ and limits to $\infty$ as $t \rightarrow b^-$.   This implies that $e^{\frac{2f}{n-2}} h$ is a complete metric  of type (I). 

When we have a type (II) or (III) warped product, we also must modify $\alpha$ so that $f$ extends to a smooth function in the polar coordinate neighborhood.  To do this, choose $\alpha$ such that  
\[ \mathrm{Ric}_h\left( \frac{\partial}{\partial t},\frac{\partial}{\partial t} \right) - \mathrm{Ric}_h\left( X,X \right) = o\left( \frac{1}{n-2} - \alpha\right) \]
as  $t \rightarrow a$ or $b$, 
and so that  the same condition holds for all derivatives.    

In the type (III) case this gives us the desired complete metric.  In the type (II) case we obtain a complete metric by controlling the asymptotics of $\alpha$ in same way as in the type (I) case. 
\end{proof} 

\begin{remark}  As this construction shows, the function $\alpha(t)$ is not unique. \end{remark}

\begin{remark}
If $n=3$,   it is possible to have solutions in the $f=f(t)$ case such that $\alpha$ is a function of 
both $t$ and $x$.  By  Schur's lemma applied to $g_N$, these examples are not possible in dimension above three. 

For example, let $\Sigma$ be any surface with Gauss curvature $\mu(x)$.  The metric
\[h = dt^2 + \cosh^2(t) g_\Sigma\]
with $f(t)=\int_0^t \frac{dr}{\sqrt{\cosh(r)}}$ and $\alpha = 1+\frac{1+\mu(x)}{\cosh(t)}$ satisfies (\ref{eqn_Ric_h}).
Moreover, $f$ is bounded, so $g=e^{2f}h$ is complete, provided $\Sigma$ is chosen to be complete.
\end{remark}

\subsection{The $f(x)$ case}
\label{sec_f(x)}
In this section we construct non-Einstein examples in the  $f=f(x)$ case with dimension $n \geq 4$.  The approach is to begin with a
metric $g_N$ on an $(n-1)$-manifold $N$ satisfying
\[\mathrm{Ric}_{g_N} = \left( \frac{1}{n-2} -\alpha \right) df \otimes df + Pg_N\]
for some function $f$ on $N$ and constants $\alpha \neq \frac{1}{n-2}$ and $P$.

\begin{example} 
\label{example_product}
The simplest example with $P$ constant is a product
manifold with $g_N = dy^2 + g_F$ where $g_F$ is an Einstein
$(n-2)$-manifold with Einstein constant $P$.  If $P=0$,
then we obtain a Ricci-flat metric so $f$ is constant.  However, if
$P<0$ we may choose $\alpha=0$ to obtain a solution with
$f$ a linear function of $y$. We also obtain solutions when $P>0$ by
letting $\alpha> \frac{1}{n-2}$.

We point out that one may obtain examples with $f$ bounded in the case
$P\neq 0$ by choosing $\alpha=\alpha(y)$
appropriately.  In particular, if $g_F$ is chosen to be complete, one
may find complete GQE
metrics $g=e^{\frac{2f}{n-2}} h$ in the $f=f(x)$ case using the
construction below.

To construct a nontrivial example in the $P=0$ case, we take the
complete metric $g_N = dy^2 + (1+y^2)g_\Sigma$,
where $\Sigma$ is a 2-sphere with constant curvature 1. The Ricci
curvature of $g_N$ is:
\[\mathrm{Ric}_{g_N} = -\frac{2}{(1+y^2)^2} dy^2.\]
Choosing $\alpha = 5/2$ and $f(y)=\arctan(y)$ ensures that
(\ref{eqn_g_N_P}) holds with $P=0$ and $n=4$.  Moreover,
since $f$ is bounded, $g=e^{f} h$ will be complete.
\end{example}
Let $Q \in \R$ and define
\[h = dt^2 + v(t)^2 g_N\]
where $v(t)$ is a nonzero solution to $v''=-\frac{Q}{n-1}v$ that is nonnegative over the range of $t$.
Direct calculation shows that
\[\mathrm{Ric}_h =  \left( \frac{1}{n-2} -\alpha \right) df \otimes df + Qg_N\]
where $v$ solves (\ref{eqn_vPQ}), restated here for convenience:
\[(v')^2 + \frac{Q}{n-1}v^2 = \frac{P}{n-2}.\]
By rescaling $g_N$, we normalize so that $P=n-2, 0, $ or $-(n-2)$.
There are six cases, up to an affine change of variable in $t$:
\begin{enumerate}
\item If $P = n-2$, there are three cases:
\begin{enumerate}
\item $v(t) = \sin(t), Q= n-1$, 
\item $v(t) = t, Q=0$, and
\item $v(t) = \sinh(t), Q= -(n-1).$
\end{enumerate}

\item If $P = 0$, there are two cases:
\begin{enumerate}
\item $v(t) = 1, Q=0$ and 
\item $v(t) = e^t, Q= -(n-1)$.
\end{enumerate}

\item If $P = -(n-2)$, there is only one case:
\begin{enumerate}
\item $v(t) = \cosh(t), Q= -(n-1)$.
\end{enumerate}

\end{enumerate}
\subsection{Conformal changes}
For the examples constructed in the last two subsections, we show that  around every point the  metric admits  local non-homothetic conformal changes preserving the GQE structure.

\begin{proposition} \label{prop_confchange_existence}
Let $h_1 = dt^2 + u'(t)^2 g_N$, and suppose that there are functions $f,$ $Q_1,$ and $\alpha$ such that 
\begin{equation}
\label{eqn_Ric_h1}
 \mathrm{Ric}_{h_1} = \left( \frac{1}{n-2} -\alpha \right) df \otimes df + Q_1h_1.
\end{equation}
Then $h_2 = u^{-2} h_1$ satisfies
\begin{equation}
\label{eqn_Ric_h2}
 \mathrm{Ric}_{h_2} = \left( \frac{1}{n-2} -\alpha \right) df \otimes df + Q_2h_2,
\end{equation}
where $Q_2 = (n-1) \left(  \frac{Q_1}{n-1}u^2 + 2 u'' u  - (u')^2 \right) $. 
\end{proposition}

\begin{proof}
This essentially follows from the observation that the steps in the proof of Lemma \ref{lemma_local_form_h} can be reversed. 
Namely, we know that, for the metric $h_1$, we have 
\[ \mathrm{Hess}_{h_1} u = \frac{\Delta_{h_1} u}{n} h_1.\]
Therefore, the formula for the change of the Ricci tensor tells us that: 
\begin{eqnarray*}
\mathrm{Ric}_{h_2} &=& \mathrm{Ric}_{h_1} +  (n-2)  \frac{\mathrm{Hess}_{h_1} u }{u}  + \left( u^{-1}\Delta_{h_1} u  - (n-1) |\nabla u|_{h_1}^2 u^{-2} \right) h_1\\
&=&  \mathrm{Ric}_{h_1}  + (n-1) \left( \frac{2}{n} u^{-1}  \Delta_{h_1} u -  |\nabla u|_{h_1}^2 u^{-2} \right) h_1  \\
&=& \left( \frac{1}{n-2} -\alpha \right) df \otimes df + (n-1) \left( \frac{ Q_1}{n-1}u^2 + \frac{2}{n}  u \Delta_{h_1} u  -  |\nabla u|_{h_1}^2 \right) h_2.\end{eqnarray*}
The formula then follows from  $\Delta_{h_1} u = n u''$ and  $|\nabla u|_{h_1}^2= (u')^2$ on $h_1$. 
\end{proof}

As a corollary we obtain conformal diffeomorphisms between the generalized quasi-Einstein manifolds constructed in the previous subsections.  
\begin{corollary} 
\label{cor_GQE_conformal}
Let $(U,e^{\frac{2f}{n-2}} h, f, \alpha, \lambda_1)$ be a GQE manifold
with $h_1 = dt^2 + u'(t)^2 g_N$ for some function $u(t)$. Then there is a function $\lambda_2$ such that  
$(U, u^{-2} e^{\frac{2f}{n-2}}h, f, \alpha,  \lambda_2)$ is also a GQE manifold.
\end{corollary}

\begin{proof} Set  $h_1 = h$, $g_1 = e^{\frac{2f}{n-2}} h_1$,  $h_2 = u^{-2} h_1$, and $g_2 =e^{\frac{2f}{n-2}} h_2$.  
From Proposition  \ref{prop_Ric_h}  we know that (\ref{eqn_Ric_h1}) holds, with
\begin{eqnarray*}
Q_1 &=& \frac{1}{n-2}\left( \Delta_{g_1} f - |\nabla f|_{g_1}^2 + (n-2)\lambda_1 \right)e^{\frac{2f}{n-2}}\\
&=& \frac{ \Delta_{h_1} f}{n-2} + \lambda_1 e^{\frac{2f}{n-2}}.
\end{eqnarray*}
Proposition \ref{prop_confchange_existence} then implies that (\ref{eqn_Ric_h2}) holds,
with
\begin{eqnarray*}
Q_2& =& (n-1) \left( u^2 \frac{Q_1}{n-1} + 2 u'' u  - (u')^2 \right) \\
&=& \frac{u^2\Delta_{h_1} f }{n-2}+  \lambda_1 e^{\frac{2f}{n-2}} u^2 +  (n-1) \left(2 u'' u  - (u')^2 \right).
\end{eqnarray*}
Direct computation now shows:
\[\lambda_2 = \left(  \lambda_1 e^{\frac{2f}{n-2}} u^2 + u h(\nabla u, \nabla f)   +  (n-1) \left(2 u'' u  - (u')^2 \right) \right)e^{\frac{-2f}{n-2}}.\]
Note that the term  $u h(\nabla u, \nabla f)$ vanishes in the $f=f(x)$ case and   equals $u u' f'$ in the $f=f(t)$ case.
\end{proof}

We can also see that $h_2$ is isometric to a warped product over a one-dimensional base with the same fiber as $h_1$:
see the proof of Theorem \ref{thm_local_classification} in section \ref{sec_conf_changes_global}.

Also note that we have, in fact, constructed a one-parameter family of conformal changes, as we can choose $u$ to be any 
anti-derivative of the warping function.  The next elementary example (which also appears in \cite{KR09})
shows that the choice  of  anti-derivative does  impact the behavior of the conformally changed metric.

\begin{example} \label{ex_sphere} Suppose that we have the standard round metric on $S^n$:
 \[ dr^2 + \sin^2(t) g_{S^{n-1}}. \]
 Then we can choose $u(t) = c - \cos(t)$.  The choice $c>1$ gives a function $u$ which is positive everywhere on the sphere, and the conformally changed metric will also be a round sphere (of possibly different curvature).  When $c\leq 1$,  $u$ will not be positive, so we do not have a global conformal change.  However, when $c=1$ we obtain stereographic projection  from the sphere minus a point to Euclidean space.  When $0<c<1$ we obtain a conformal change from a portion of the sphere to a portion of hyperbolic space,
possibly rescaled. 
 \end{example} 

We also note that with a rotationally-symmetric metric, (i.e. if  $u'$ vanishes somewhere), it is always possible to choose 
the conformal factor $u$ to be positive everywhere: since $u'(t)>0$, $u$ is bounded from below and thus can be made 
positive by adding a suitable constant.

\section{Conformal  fields }
\label{sec_conf_fields}
We prove local and global classification results for GQE manifolds admitting conformal fields.
In this section we assume:
\begin{enumerate}
 \item $(M,g,f,\alpha,\lambda)$ is a GQE manifold with $\alpha \neq \frac{1}{n-2}$, 
 \item $V$ is a vector field on $M$ such that $L_V g = 2\eta g$, with $\eta$ non-constant (i.e., $V$
is a non-homothetic conformal field), and
 \item $V$ preserves the GQE structure, in the sense that $D_V f$ equals a constant $c$, and $D_V \alpha =0$.
\end{enumerate}
First, note that $V$ is also a non-homothetic conformal field for $h= e^{-\frac{2f}{n-2}}g$:
\[ L_V h = 2 \left( \eta - \frac{c}{n-2}\right) h.\]
We define $\sigma = \eta - \frac{c}{n-2}$.

Next, we make the following observations. If $\phi_t$ is the local flow of $V$ about some point, then
$\phi_t^* g = w_t^{-2} g$ for a smooth family of functions $w_t$.  
The smooth family $u_t = w_t e^{\frac{C(t)}{n-2}}$ (where
$\phi_t^* f = f + C(t)$) satisfies $\phi_t^* h = u_t^{-2} h$, and therefore 
solves (\ref{eqn_hessian}) for each $t$ with respect to $h = e^{-\frac{2f}{n-2}}g$, by the proof of Lemma
\ref{lemma_local_form_h}.  
Differentiating in $t$, we find that $\eta$ satisfies equation (\ref{eqn_hessian})
on $M$ \footnote{An alternative approach is to apply the Lie derivative with respect to $V$ to equation
(\ref{eqn_Ric_h}), making use of formula (3.2) of \cite{KR09}: $L_V \mathrm{Ric}_h = -(n-2)\mathrm{Hess}_h \sigma - \Delta \sigma \cdot h.$} (even though the local flows of $V$ may not be globally defined).
Following the same arguments as in section \ref{sec_conf_changes}, we have:
\begin{observation}
\label{sigma_u}
The non-constant function $\sigma$ satisfies (\ref{eqn_hessian}), and the local and global classification results
(Lemmas \ref{lemma_local_form_h}, \ref{lemma_local_form_f}, \ref{lemma_ft_fx} and
Theorems \ref{thm_local_classification} and \ref{thm_compact_case}) hold in the present case,
with $u$ replaced by $\sigma$.
\end{observation}
Consequently, $h$ is of the form (locally or globally)
\begin{equation}
\label{eqn_h_sigma}
 h = dt^2 + \sigma'(t)^2 g_N
\end{equation}
for $t \in I$. We similarly define $f=f(t)$ (resp., $f=f(x)$) to mean $\nabla f$ is parallel (resp. orthogonal) to $\nabla \sigma$.

We are therefore led to study conformal fields on a warped product over a one-dimensional base.  
We fix notation for $V$ by writing 
\[ V = v_0(t, x) \frac{\partial}{\partial t} +V_t, \] 
where $v_0$ is some function on $I \times N$, and $V_t$ is the projection of $V$ onto the factor $\{t\} \times N$.  
Some general facts regarding this case are collected in the statement below, which follows immediately from
Proposition \ref{prop_conf_fields_app} in appendix \ref{app_conf_fields}.
\begin{proposition}
\label{prop_conf_fields}
A vector field $V$ satisfies
\[ L_V h = 2 \sigma h,\]
with $h$ given by (\ref{eqn_h_sigma})
 if and only if 
\begin{enumerate}
\item $V_t$ is a conformal field for $g_N$ for each $t$ with $L_{V_t} g_N = 2 \omega_t g_N$,
\item  $ \frac{\partial}{\partial t}\left(\frac{v_0}{\sigma'}\right)  = \frac{\omega_t}{\sigma'}$, and 
\item $\frac{\partial V_t}{\partial t} = -\frac{1}{(\sigma')^2}\nabla^N v_0.$
\end{enumerate}
Moreover $\sigma= \frac{v_0\sigma''}{\sigma'}   + \omega_t = \frac{\partial v_0}{\partial t}.$
\end{proposition}

We consider separately the cases in which $f=f(t)$ and $f=f(x)$, with the goal of classifying the structures of $g$ 
and $V$, both locally and globally.

\subsection{$f=f(t)$ case}
\label{sec_conf_fields_f(t)}
The first observation is that $f$ is constant when  $c=0$. 

\begin{proposition}
\label{lemma_c=0}
If $f=f(t)$ and $D_V f = 0$, then $f$ is constant.
\end{proposition}

\begin{proof}
Suppose $I$ is an open interval on which $f'(t) \neq 0$.   The condition $D_V f = 0$ is equivalent 
to $v_0(t,x) f'(t) =0$, so $v_0=0$ on $I$. From Proposition \ref{prop_conf_fields}, $\sigma = \frac{\partial v_0}{\partial t}=0$ 
on $I$.  This contradicts the fact that the zeros of $\sigma$ are isolated.  
\end{proof}
The following corollary is a  special case. 

\begin{corollary}
\label{lemma_sigma_no_crit}
Suppose $\sigma$ has a critical point at $p \in M$.  Then $f$ is constant on any polar coordinate neighborhood of
$p$.
\end{corollary}

\begin{proof}
If $d\sigma(p)=0$, then $h$ admits polar coordinates about $p$ and $f=f(t)$.  
By smoothness, $df(p)=0$.  Since $D_V f$ is constant, it is identically zero. 
\end{proof}

Now we may prove Theorem \ref{thm_compact_case_conformal_fields} from the introduction, restated
below for the reader's convenience.
\begin{theorem}
\label{thm_compact_case_conformal_fields2}
Suppose $(M,g,f,\alpha,\lambda)$ is a complete GQE manifold,
with $\alpha \neq \frac{1}{n-2}$, that admits a structure-preserving non-homothetic 
conformal field: $L_V g = 2\eta g$.  If $\eta$ has a critical point (e.g., if $M$ is compact),
then $f$ is constant and $(M,g)$ is isometric to a simply-connected space form.
\end{theorem}
\begin{proof}
$(M,h)$ admits a polar coordinate neighborhood $U$ about a critical point $p$ of $\sigma$ with $f=f(t)$ on $U$, so that $g$
is rotationally-symmetric with pole $p$.  By Corollary \ref{lemma_sigma_no_crit}, $f$ is constant on $U$.
In the compact case, $U$ covers $M$ except for a point, so $f$ is constant; in the non-compact case, $U=M$,
and $f$ is constant.  Thus $g$ is a Einstein.  Complete, rotationally-symmetric Einstein manifolds are well-known
to be the simply-connected space forms.
\end{proof}

Thus, we restrict to the case in which $M$ is non-compact and $\sigma$ has no critical points; from the previous
results, we may also assume $c\neq 0$, so that $f'$ never vanishes.  We assume $h$ is of the form (\ref{eqn_h_sigma}) on $U = I \times N$, and
where $\sigma'>0$ on $I=(a,b)$.  We  have that  $v_0 = \frac{c}{f'(t)}$ and in particular,  $v_0$ is a function of 
only $t$ and never vanishes.  Corollaries \ref{cor_v_0(t)} and \ref{cor_fields_integrated_eqns} 
imply that $V_t$ is independent of $t$, $\omega:=\omega_t$ is constant, and $v_0$ and $\sigma$ solve:
\begin{eqnarray*}
\sigma &=& \sigma''\left(A + \omega \int \frac{dt}{\sigma'}\right) + \omega\\
v_0 &=& \sigma'\left( A + \omega \int \frac{dt}{\sigma'}\right).
\end{eqnarray*}
for some constant $A$.  Defining $r(t) = \int \frac{dt}{\sigma '}$, an increasing function of $t$, these equations become 
\begin{eqnarray} \label{eqn_sigma_r}  \sigma &=& \sigma''(A+\omega r) + \omega\\
\label{eqn_v_r} v_0 &=& \sigma'(A + \omega r).\end{eqnarray} 

Note that we have not yet used the GQE structure; doing so yields the following. 
\begin{lemma} 
\label{Lemma_ConfField_Local}  
Suppose $(U,h)$, $V$ and $\sigma$ are as above.
If $(U,h,f,\alpha, \lambda)$ is a GQE manifold, $V$  preserves the GQE structure, and 
$f=f(t)$, then 
\begin{enumerate}
\item $\sigma$ is a solution to  (\ref{eqn_sigma_r}) for some constants $A$ and $\omega$, 
\item $V = v_0(t) \frac{\partial}{\partial t}+ V_0$ where $v_0$ is given in terms of $\sigma$ by (\ref{eqn_v_r}) and is non-zero on $(a,b)$,  and $V_0$ is a fixed homothetic field for $g_N$ with expansion factor $\omega$, 
\item $f(t) = \int \frac{c}{v_0(t)} dt$,  and 
\item $\alpha=K_1 +  K_2 \mu(x) $ where $K_i$ is are explicit constants determined by $A$, $\omega$, $\sigma$, $c$, and $n$
(see (\ref{eqn_alpha})),  and $\mathrm{Ric}_{g_N} = \mu(x) g_N$.   
\end{enumerate}
Conversely, if $A, \omega, \sigma, V, f, \alpha, g_N, K_1, K_2$ and $c$ satisfy (1)--(4), then
$(U,h,f,\alpha, \lambda)$ is a GQE manifold with structure-preserving conformal field $V$.
\end{lemma}

\begin{remark} In particular,  $\alpha$ is constant if $n>3$.  The proof will also show that $\alpha$ is constant if $n=3$ and $\omega \neq 0$.\end{remark} 

We have already established (1)--(3) above.  Before proving (4),  we note the following fundamental fact about solutions to  (\ref{eqn_sigma_r}). 

\begin{proposition} A function $\sigma$ solves (\ref{eqn_sigma_r}) if and only if the quantity
\begin{equation}
\label{eqn_K}
K = (A+\omega r)(\sigma')^2 - \sigma(\sigma - \omega) 
\end{equation}
is constant.  
\end{proposition}

\begin{remark} When $\omega = 0$ this is the well-known fact that $A (\sigma')^2 - \sigma^2$ is constant for 
solutions to $\sigma'' = \frac{\sigma}{A}$. \end{remark}

\begin{proof} 
Differentiate with respect to $t$ and use $\frac{dr}{dt} = \frac{1}{\sigma'}$:
\begin{eqnarray*} 
\frac{d}{dt} \left( (A+\omega r)(\sigma')^2 - \sigma(\sigma - \omega) \right) &=& \omega \frac{dr}{dt} (\sigma')^2 + 2(A+\omega r)
       \sigma' \sigma'' - 2 \sigma \sigma' + \omega \sigma' \\
&=& 2 \sigma' \left( (A+\omega r) \sigma'' + \omega - \sigma \right).
\end{eqnarray*}
\end{proof}

\begin{proof}[Proof of Lemma \ref{Lemma_ConfField_Local}] 
In order to have a GQE structure,  the Ricci curvature of $h$ must be given both from the warped product formula
(\ref{eqn_Ric_h_WP}) and from (\ref{eqn_Ric_h}), leading to:
\begin{equation}
\label{eqn_sigma_f}
 \left(\frac{1}{n-2}-\alpha\right)f'(t)^2= -(n-2)\left( \frac{\sigma'''(t)}{\sigma'(t)} - \frac{\sigma''(t)^2}{\sigma'(t)^2}\right)  - \frac{\mu(x) }{\sigma'(t)^2}  ,
\end{equation}
We rewrite (\ref{eqn_sigma_r}) as 
\[ \sigma'' = \frac{\sigma- \omega}{A+ \omega r}. \]
Differentiating this equation with respect to $t$ yields:
\[ \frac{\sigma'''}{\sigma'} = \frac{1}{A+\omega r} - \frac{\omega \sigma''}{(\sigma')^2(A+\omega r)} \]
Substituting into formula (\ref{eqn_sigma_f}) gives
\begin{eqnarray*}
\left(\frac{1}{n-2}-\alpha\right)f'(t)^2 &=& -(n-2) \left( \frac{\sigma'''}{\sigma'} - \frac{(\sigma'')^2}{(\sigma')^2} \right) - \frac{\mu(x) }{\sigma'(t)^2} \\
&=&  -(n-2) \left( \frac{(\sigma')^2(A + \omega r) - \sigma(\sigma - \omega) }{(\sigma')^2(A+ \omega r)^2}\right)- \frac{\mu(x) }{\sigma'(t)^2}. 
\end{eqnarray*}
On the other hand, 
\[ f'(t) = \frac{c}{v_0 }= \frac{c}{ \sigma'(A+\omega r)}, \]
and  $(A+\omega r)(\sigma')^2 - \sigma(\sigma - \omega) = K$ is constant, implying
\[ \alpha = \frac{1}{n-2} + \frac{(n-2)K + \mu(x)(A+ \omega r)^2}{c^2}.\]
However, if $\omega \neq 0$, in order for $g_N$ to admit a non-Killing homothetic field, it must be flat 
(cf. pg. 242 of \cite{KN}).  Therefore we have:
\begin{equation}
\label{eqn_alpha}
\alpha = \left\{ \begin{array}{cc} \frac{1}{n-2} + \frac{(n-2)K + \mu(x)A^2}{c^2},  &  \omega = 0 \\\\  \frac{1}{n-2} + \frac{(n-2)K}{c^2}, & \omega \neq 0. \end{array} \right.
\end{equation}
\end{proof}
 
We separately analyze the cases in which $\omega = 0$ and $\omega \neq 0$.  If $\omega = 0$, then $A \neq 0$,
and the possible solutions to (\ref{eqn_sigma_r}) are (up to shifting $t$ and rescaling $V$ and $\sigma$):
$\sigma(t) = \cos(\kappa t)$, $\sigma(t)=e^{\kappa t}$, $\sigma(t)=\sinh(\kappa t)$, or $\sigma(t)=\cosh(\kappa t)$, 
where $\kappa=\sqrt{\frac{1}{|A|}}$.  These all produce local examples.

We are interested in determining when it is possible to construct an example with $g$ complete.  
To simplify notation, we assume $\kappa=1$.

\begin{example}
\label{example_cosh}
Suppose $\sigma(t) = \sinh (t)$, and 
\begin{align*}
h &= dt^2 + \cosh(t)^2 g_N,\\
f &= \int_0^t \frac{dr}{\cosh(r)},\\
V &= \cosh(t)\frac{\partial}{\partial t} + X,
\end{align*}
where $N$ is any complete space with $\mathrm{Ric}_{g_N} = \mu g_{N}$, with Killing field $X$ (possibly zero).  One can 
readily check that $V$ is a conformal field for $g= e^{\frac{2f}{n-2}} h$ (with $\omega=0$) with expansion factor $\eta=\sinh(t) + \frac{1}{n-2}$, 
that $D_V f=1$, and that $g$ is complete (since $f$ is bounded).  Moreover, choosing $\alpha$ so that
\[\alpha = \frac{1}{n-2} + n-2 + \mu\]
assures that $(M,g,f,\alpha, \lambda)$ is a GQE manifold for some $\lambda$.
\end{example}

\begin{example}
A similar example occurs with $\cosh(t)$ replaced with $e^t$ and
\[\alpha = \frac{1}{n-2} +\mu.\]
However, in this case, $f$ is given by $-e^{-t}$ (up to a constant), and the conformal metric $g= e^{\frac{2f}{n-2}} h$ is necessarily incomplete.
\end{example}

\begin{example}
Suppose $\sigma(t)=\cosh(t)$, so that $\sigma$ has a critical point at $t=0$.  Then $h$ is defined only on $(0,\infty)$
(or its negative), and the arc length with respect to $g$ is given up to constants by
\[ s(t) = \int_1^t \exp\left(\frac{c}{n-2} \int_1^z \frac{dy}{\sinh(y)}\right) dz.\]
However, $\lim_{t \to 0^-} s(t)$ is finite, so that $g$ is incomplete.  A similar argument applies if $\sigma(t)=\cos(t)$.
\end{example}

Next, we move on to the case in which $\omega \neq 0$. Perform the change of variables $r = \int_0^t \frac{dt}{\sigma'(t)}$.  
Since $ \frac{d\sigma}{dr} = \left( \frac{d\sigma}{dt} \right)^2,$  (\ref{eqn_K})  becomes
\begin{eqnarray*}
\frac{d\sigma}{dr} = \frac{K + \sigma(\sigma-\omega)}{A+\omega r}.
\end{eqnarray*}
Separating variables and completing the square produces:
\begin{eqnarray*} 
\int \frac{d\sigma}{ K - \frac{\omega^2}{4} +(\sigma - \frac{\omega}{2})^2} &=& \int \frac{dr}{A+\omega r}  \\
&=& \frac{1}{\omega} \ln|C(A+\omega r)|,
\end{eqnarray*}
for some constant $C>0$. 
Let $B = K - \frac{\omega^2}{4}$.  There are three cases depending on the sign of $B$. 
\begin{itemize}
\item If $B=0$ then $\sigma(r) - \frac{\omega}{2} = \frac{-\omega}{\ln|C(A + \omega r)| }$. 
\item If $B>0$ then $\sigma(r)  - \frac{\omega}{2} = \sqrt{B} \tan \left( \sqrt{B}\ln |C(A + \omega r)|  \right)$.
\item If $B<0$ then 
\begin{align*}
\sigma(r)  - \frac{\omega}{2} &= \sqrt{-B} \tanh \left( \sqrt{-B}\ln |C(A + \omega r)|  \right)\\
 &= \sqrt{-B} \frac{|C(A+\omega r)|^{2\sqrt{-B}}-1}{|C(A+\omega r)|^{2\sqrt{-B}}+1}.
\end{align*}

\end{itemize}

Computing the derivative of $f$ with respect to $r$ using $f'(t)=\frac{c}{v_0(t)}$ and (\ref{eqn_v_r}) gives
\[ \frac{df}{dr} = \frac{c}{A+\omega r}.\] So $f(r) = c\ln(D|A+\omega r|)$ for a constant $D>0$.  

Thus, we have completely determined the local structure of $g,f$ and $V$ in the $\omega \neq 0$ case.
Conversely, given constants $\omega \neq 0$, $c\neq 0$, $C>0, D>0$, $A,B$, we can use the above formulas for 
$f(r)$ and $\sigma(r)$ to construct local examples; the parameter $t$ may be recovered by 
$t(r) = \int \sqrt{\sigma'(r)} dr$.  Next, we are interested in analyzing which of these examples is complete.

We begin with a function $\sigma(r)$ of one of the three forms above, defined on a maximal interval $I$ 
such that $\frac{d \sigma}{dr}>0$. 
To simplify calculations we assume that $\omega=1$ by rescaling $V$ and $\sigma$; $A=0$ by shifting $s$; 
$D=1$ by shifting $f$; and $r>0$ by symmetry. In each of the following cases, $f(r)= c\ln(r)$.
\begin{itemize}
\item  If $B=0$, then
\begin{eqnarray*}
\sigma(r)&=& \frac{1}{2} - \frac{1}{\ln(Cr)}, \\
\frac{d \sigma}{dr} &=& \frac{1}{r \ln(Cr)^2}.
\end{eqnarray*}
$\sigma$ is undefined at $r=1/C$, so we consider $I=(0,1/C)$ or $(1/C, \infty)$.  
The arc-length parameter for $g$ is given by 
\begin{eqnarray*}
s(r)=\int e^{\frac{f}{n-2}} dt &=& \int  r^{\frac{c}{n-2}}  \left( \frac{1}{r^{1/2}\ln(Cr)}\right) dr \\
&=& \int \frac{r^{\frac{2(c+1)-n}{2(n-2)}}}{\ln(Cr)} dr.
\end{eqnarray*}
By analyzing the limiting behavior of $s(r)$ at $r=0^+, 1/C^{\pm}$ and $\infty$, we find that $g$ is complete
with $I=(0,1/C)$ if and only if $c \leq -\frac{n-2}{2}$ and 
with $I=(1/C, \infty)$ if and only if $c \geq -\frac{n-2}{2}$.

\item If $B>0$, then 
\begin{eqnarray*}
\sigma(r)&=& \frac{1}{2} + \sqrt{B}\tan(\sqrt{B}\ln(Cr)), \\
\frac{d \sigma}{dr} &=& \frac{B\sec^2(\sqrt{B}\ln(Cr))}{r}.
\end{eqnarray*}
We take the interval $I=(\frac{1}{C}e^{\frac{-\pi/2+k}{\sqrt{B}}},\frac{1}{C}e^{\frac{\pi/2+k}{\sqrt{B}}})$ for 
any integer $k$. $t$ is given by 
\[ t = \int \frac{\sqrt{B}\sec(\sqrt{B}\ln(Cr))}{\sqrt{r}}dr, \]
which implies that $t$ is defined on $(-\infty, \infty)$.  Since $f$ is bounded in this case, $g$ is complete. 

\item If $B<0$, then 
\begin{eqnarray*}
\sigma(r) &=& \frac{1}{2}+\sqrt{-B} \frac{|Cr|^{2\sqrt{-B}}-1}{|Cr|^{2\sqrt{-B}}+1},\\
\frac{d \sigma}{dr} &=& \frac{-4B|Cr|^{2\sqrt{-B}-1}}{(|Cr|^{2\sqrt{-B}}+1)^2},
\end{eqnarray*}
and we take $I=(0,\infty)$. The arc-length with respect to $g$ is:
\begin{equation}
\label{eqn_arc_length}
s(r)=\int \frac{2\sqrt{-B}r^{\frac{c}{n-2}} (Cr)^{\sqrt{-B}-1/2}}{(Cr)^{2\sqrt{-B}}+1}.
\end{equation}
For no values of $B<0,C>0,c\neq 0$ does $|s(r)|$ limit to infinity at $r=0$ and $r=\infty$; 
thus metrics of this form are incomplete.
\end{itemize}

At this point we have a full understanding of the $f=f(t)$ case, in both the $\omega=0,\omega\neq0$ subcases.
For future reference, we analyze the completeness of the conformal field $V$.
\begin{lemma}
\label{lemma_completeness}
If $(M,g)$ is complete and non-compact in the $f=f(t)$ case, then the conformal field $V$ is not complete.
\end{lemma}

\begin{proof}
Suppose $V$ is complete.  By Lemma \ref{Lemma_ConfField_Local}, $V$ is of the form $V = v_0(t) \frac{\partial}{\partial t}+ V_0$, 
where $V_0$ is a fixed homothetic field for $g_N$.  If $g$ is complete, so is $g_N$.  It follows that $V_0$ is a complete field
(see p. 234 of \cite{KN}) on $N$ and extends naturally to a complete vector field on $M$.  Then 
$V-V_0 = v_0(t) \frac{\partial}{\partial t}$ is a complete field on $M$ and therefore on $\R$.  We analyze the two cases.

If $\omega=0$, then $v_0(t) \frac{\partial}{\partial t}$ = $\cosh(t)\frac{\partial}{\partial t}$, 
which is not a complete field on $\R$ by elementary considerations.  This is a contradiction.

If $\omega \neq 0$, we can write $v_0 \frac{\partial}{\partial t}$ as $(A+\omega r) \frac{\partial}{\partial r} 
= r \frac{\partial}{\partial r}$.  The flow of this vector 
field at time $\epsilon$ is given by scaling $r$ by $e^\epsilon$.  Such flows are globally well-defined only on 
$\R$, $(-\infty, 0)$, and $(0, \infty)$.  None of the complete examples we considered above were defined on such a subset,
again leading to a contradiction.
\end{proof}

\subsection{$f=f(x)$ case}
\label{sec_conf_fields_f(x)}
We also analyze the case of a conformal field in the $f=f(x)$ setting, so that $n \geq 4$ and $Q$ is constant
(by Observation \ref{sigma_u} and Theorem \ref{thm_local_classification}).  Without loss of generality, assume $f$ is non-constant.  We prove:
\begin{proposition}
\label{prop_f(x)_cases}
If the $f=f(x)$ case occurs, then 
$\sigma$ solves
\begin{equation}
\label{eqn_ode_sigma}
\sigma''' = -\frac{Q}{n-1} \sigma'.
\end{equation}
Moreover, $D_{V_t} f = c$ and $D_{V_t} \alpha=  0$, and either:
\begin{enumerate}
 \item $V_t$ is a non-homothetic conformal field on $g_N$ for some $t$, or else
 \item $V_t$ is a Killing field on $g_N$, independent of $t$, $v_0=v_0(t)$ is a constant multiple of $\sigma'(t)$,
and $\sigma'(t)$ is non-constant.
\end{enumerate}
\end{proposition}

Examples of case (2) are found using section \ref{sec_f(x)}. After the proof, we demonstrate that case (1) may occur as well.
\begin{proof}
First, by Lemma \ref{lemma_local_form_f} and Observation \ref{sigma_u},
$\sigma$ satisfies (\ref{eqn_ode_sigma}) and $\alpha=\alpha(x)$.
Additionally, since $D_V f = c$ and $f=f(x)$, we have $D_{V_t} f = c$ for each $V_t$, and 
likewise $D_{V_t} \alpha = D_V \alpha = 0$.  

By Proposition \ref{prop_conf_fields}, for each $t$, $V_t$ is a conformal field 
with expansion factor $\omega_t(x)$ on $N$.  If any $\omega_t(\cdot)$ is non-constant on $N$, 
$V_t$ is non-homothetic on $N$, and we are in case (1).

Otherwise, $\omega$ depends only on $t$.  By (\ref{eqn_ode_sigma}), if $\sigma'$ is constant, then 
$Q=0$.  
If $\omega_t \neq 0$ for some $t$, then $g_N$ admits a homothetic field that is non-isometric and so $g_N$ is flat.  
Combining (\ref{eqn_Ric_h}) and (\ref{eqn_Ric_h_WP}) shows that $\left(\frac{1}{n-2}-\alpha\right)df \otimes df$
is pointwise proportional to $g_N$.  By comparing rank, it follows (since $\alpha \neq \frac{1}{n-2})$ that $f$ is constant, a contradiction.
We conclude that $\omega_t$ is identically zero.  But Proposition \ref{prop_conf_fields} 
and the constancy of $\sigma'$ imply $\sigma \equiv 0$, a contradiction.

Thus, we may assume $\sigma'$ is non-constant, so that
$\sigma''(t)$ vanishes only for isolated $t$ by (\ref{eqn_ode_sigma}).
From $\sigma= \frac{v_0}{\sigma'} \sigma''  + 
\omega_t$ of Proposition \ref{prop_conf_fields}, we see that $v_0=v_0(t)$. Then by Corollary \ref{cor_v_0(t)},
$V_t$ is independent of $t$ and $\omega_t=\omega$ is constant in $t$ and $x$.
If $\omega\neq 0$, $g_N$ admits a homothetic field that is not Killing, and the same argument
as above leads to a contradiction.  Thus $\omega=0$, and 
Proposition \ref{prop_conf_fields} implies $v_0$ is a constant multiple of $\sigma'(t)$.
\end{proof}

We demonstrate that the first case of Proposition \ref{prop_f(x)_cases} can occur, at least locally.
\begin{example}
\label{ex_varying_conf_field}
Suppose that
$(K,h_K)$ is some $n$-manifold ($n \geq 3$) satisfying
\[\mathrm{Ric}_{h_K} = \left(\frac{1}{n-2}-\alpha_K\right) df \otimes df - (n-1) h_K\]
for a non-constant function $f:K \to \R$ and some constant $\alpha_K$ (cf. Example \ref{example_product}).  
Define $M = \R^2 \times K$ with coordinates $(t,y)$ on $\R^2$ and metric 
\[h_M = dt^2 + \cosh^2 t \left(dy ^2+\cosh^2 y\, h_K\right),\]
which satisfies
\[\mathrm{Ric}_{h_M} = \left(\frac{1}{n}-\alpha_M\right) df \otimes df - (n+1) h_M\]
for an appropriate constant $\alpha_M$. Note the $t$-level sets each admit a conformal field
$\cosh(y) \frac{\partial}{\partial y}$ with expansion factor $\sinh(y)$.  We  define
a vector field $V$ on $M$ by:
\[V = \left(\cosh(t)\sinh(y)\int_0^t \frac{\varphi(z)}{\cosh(z)}dz\right)
\frac{\partial}{\partial t} + \varphi(t) \cosh(y) \frac{\partial}{\partial y},\]
where
\[\varphi(t)=\sin(2\arctan(\tanh(t/2))).\]
Direct calculation (using Proposition \ref{prop_conf_fields_app}) shows that $V$ is a conformal field of $h$ with expansion
factor $\sigma(t,y) = \sinh(y)\left(\sinh(t) \int_0^t \frac{\varphi(y)}{\cosh(y)}dr + \varphi(t)\right)$;
its restriction $V_t$ to a level set of $t$
is a conformal field with expansion factor $\omega_t(y)=\varphi(t)\sinh(y)$; in particular, $V_t$ is non-homothetic for almost all $t \in \R$.

Moreover, $V$ preserves the GQE structure of $(M, e^{\frac{2f}{n}} h_M)$: $D_V f=0$, since $f$ is a function on $K$, and $D_V \alpha_M = 0$ since $\alpha_M$ is constant.
\end{example}

\begin{remark}
In the above example, $h_M$ admits a warped product structure with respect to the level sets of $\sigma$, by our
classification theorem.  However, we point out this
structure is not apparent from the expression of $h_M$ in coordinates $t, y$.
\end{remark}

\subsection{Complete conformal fields}
Here we prove the generalization of the theorem of Yano and Nagano on complete conformal fields on Einstein spaces stated
in the introduction.

\begin{proof}[Proof of Theorem \ref{thm_complete_case_conformal_fields}]
Suppose $(M,g,f,\alpha,\lambda)$ is a complete GQE manifold equipped with a structure-preserving non-homothetic
conformal vector field $V$: $L_V g = 2\eta g$.  Assume $V$ is complete.  If $\eta$ has a critical point, then by
Theorem \ref{thm_compact_case_conformal_fields2}, $(M,g)$ is a space form and $f$ is constant.  However, the round sphere is the only
space form admitting a complete non-homothetic conformal field.

Otherwise, $M$ is non-compact, $\sigma=\eta-\frac{c}{n-2}$ has no critical points, and the work of sections \ref{sec_conf_fields_f(t)} and \ref{sec_conf_fields_f(x)} applies.
If $f=f(t)$, then Lemma \ref{lemma_completeness} implies that $V$ is incomplete, a contradiction. Thus $f=f(x)$ on $M$.

In this case, since $g$ is complete, Theorem \ref{thm_local_classification} and Observation \ref{sigma_u} imply that $h$ 
is  one-dimensional warped product  defined for all $t \in \R$.  Since $\sigma'$ solves (\ref{eqn_ode_sigma}) 
and has no zeros we conclude that $\sigma'(t)$ is (up to an overall scaling of $\sigma$ and of $V$, 
and a translation of $t$), 
equal to $1$, $e^{\kappa t}$ or $\cosh(\kappa t)$, where $\kappa = \sqrt{\frac{-Q}{n-1}}$.    
If $\sigma' \equiv 1$, then by Proposition \ref{prop_conf_fields}, $\sigma' = \omega_t$, which implies $\omega_t$
depends only on $t$.  This contradicts part (2) of Proposition \ref{prop_f(x)_cases}.

On the other hand, the following argument, which is an adaptation of  Yano-Nagano's argument in the Einstein case, shows that $\nabla \sigma$ must  be a complete field if $V$ is complete.  This is a contradiction, since $e^{\kappa t} \frac{\partial}{\partial t}$  and $\cosh(\kappa t) \frac{\partial}{\partial t}$ are not complete fields on $\mathbb{R}$. 

Computing the Laplacian on the warped product $h$ gives:
\begin{align*}
 L_{\nabla \sigma} h &= 2\frac{\Delta \sigma}{n} h\\
&= 2\sigma'' h.
\end{align*}
However, by Proposition \ref{prop_f(x)_cases}, $\sigma'' = -\frac{Q}{n-1}\sigma + c_0$, where $Q$ and $c_0$ are constants.
In particular, $W=\frac{Q}{n-1} V + \nabla \sigma$ satisfies $L_W h = 2 c_0 h$.
In the metric $g$,
\begin{align*}
 L_W g = \left(\frac{2D_W f}{n-2} + c_0\right)g.
\end{align*}
However, $D_W f$ is constant, since $D_V f$ is constant and $\nabla \sigma$ is orthogonal to $\nabla f$.  It
follows that $W$ is a homothetic field for the complete metric $g$, and so $W$ is complete (see p. 234 of \cite{KN}).
Since the set of complete conformal fields on a Riemannian manifold forms a Lie algebra, $\nabla \sigma$ is complete, a 
contradiction to the form of $\sigma$.
\end{proof}

\section{Gradient Ricci solitons and $m$-quasi-Einstein metrics}
\label{sec_solitons}
In this section we specialize to the case where $\alpha$ and $\lambda$ are constant, first obtaining some rigidity for 
the function $Q$ of Proposition \ref{prop_Ric_h}.   

\begin{proposition} 
If $(M,g,f)$ is a complete gradient Ricci soliton or a complete $m$-quasi-Einstein manifold, 
then $Q$ is constant if and only if $f$ is constant. 
\label{prop_Q_f}
\end{proposition}
\begin{proof}
First note that if $f$ is constant, the same is true for $Q$ by definition.  Now we prove the converse.

\paragraph{\emph{Gradient Ricci soliton case:}}
If $g$ is a gradient Ricci soliton ($\alpha=0$ and $\lambda$ constant)  we have the following formula due to Hamilton (see Proposition 1.15 of \cite{Chowetc} for a proof):
\begin{equation}
\label{Sol_int}  \Delta f - |\nabla f|^2 =  -2\lambda f + c, 
\end{equation}
for some constant $c$.  Plugging into the formula for $Q$ (\ref{Qeqn}), we obtain 
\[ Q = \frac{1}{n-2} \left( - 2 \lambda f + c +(n-2) \lambda \right) e^{\frac{2f}{n-2}}. \]
From this, one can see that if $dQ$ vanishes identically and $df \neq 0$ at some point, then
$\lambda = c = 0$.  Then we have 
\[ \Delta f - |\nabla f|^2 = 0 \]
Moreover $\Delta f = - R$ from the trace of the soliton equation, where $R$ is the scalar curvature of $g$.
In particular,
\[ - R - |\nabla f|^2 = 0. \]
However, Chen has shown that if $\lambda =0$ then  $R \geq 0$ \cite{Chen} (cf. \cite{Yokota}, \cite{Zhang}),
implying $R=|\nabla f| = 0$ when $c=0$.

\vspace{5mm}
\paragraph{\emph{$m$-quasi-Einstein case:}}
If $g$ is $m$-quasi-Einstein ($m>0$, $\alpha = \frac{-1}{m}$ and $\lambda$ constant) we have the 
equation proven by Kim--Kim  \cite{KimKim} that 
\[ \Delta f - |\nabla f|^2 = m \left( \lambda - \mu e^{\frac{2f}{m}}\right) \] for some constant $\mu$,
which gives
\[ Q = \frac{1}{n-2} \left( (n+m-2) \lambda  - \mu e^{\frac{2f}{m}} \right) e^{\frac{2f}{n-2}}.\]
If $dQ$ vanishes identically and $df\neq 0$ at some point, then $\lambda= \mu = 0$.
By a result of Case, $f$ is constant \cite{Case_nonexistence}.
\end{proof}

\begin{corollary}
\label{cor_soliton_f(t)}
Suppose $(M,g,f)$ is a complete gradient Ricci soliton or a complete $m$-quasi-Einstein manifold.  If $(M,g,f)$
admits a non-homothetic structure-preserving conformal diffeomorphism or conformal field, then 
only case (\ref{eqn_g_f(t)}) in Theorem \ref{thm_local_classification} may occur.
\end{corollary}
\begin{proof}
In the $f=f(x)$ case, $Q$ is constant by Lemma \ref{lemma_local_form_f} (and Observation \ref{sigma_u}
in the case of a conformal field).  Then $f$ is constant, so we may say without loss of generality
that $f=f(t)$.
\end{proof}

When the constant $\lambda$ is nonnegative, we also have the following. 

\begin{proposition} \label{prop_nosplitting}
If a complete gradient Ricci soliton or complete $m$-quasi-Einstein metric $(M,g)$ of the form (\ref{eqn_g_f(t)}),
\[g=  ds^2 + v(s)^2g_N,\]
with $g_N$ Einstein has $\lambda \geq 0$, then  either $g$ is rotationally-symmetric (on $\mathbb{R}^n$ or $S^n$),
or $v$ is constant and $g$ is the product metric on $\mathbb{R} \times N$. 
\end{proposition}

\begin{proof}
The result  follows from the work of various authors.  
The main observation is that a complete metric of type (I) (see Definition \ref{def_wp}) of the form 
(\ref{eqn_g_f(t)}) contains a line in the $s$-direction: a geodesic defined on $(-\infty, \infty)$ that is minimizing 
on all its sub-segments. 

In the $m$-quasi-Einstein case,  a version of the Cheeger--Gromoll splitting theorem holds if $\lambda \geq 0$  \cite{FLZ}.    Therefore, if  $g$ is not a product it then must be 
rotationally-symmetric (i.e., type (II) or (III)).   In fact, if $\lambda>0$, $M$ must be compact \cite{Qian}.

If $(M,g)$ is a gradient Ricci soliton, we may, without loss of generality, replace  $g_N$ with
a space form of the same dimension and  with the same Einstein constant.  In particular, $(M,g)$ is now locally conformally flat.
Locally conformally flat gradient Ricci solitons with $\lambda \geq 0$ are classified \cite{CaoChen}, 
however, we do not need the entire argument in this special case.  Indeed, by the work of 
Chen \cite{Chen}  and Zhang \cite{ZZhang}  (cf. Proposition 2.4 of \cite{CaoChen})  
a locally conformally flat gradient Ricci soliton with $\lambda \geq 0$  either has positive curvature operator or is a product.  
However, by the classical splitting theorem of Toponogov, a space with positive curvature cannot contain a 
line, so a type (I) gradient Ricci soliton with $\lambda \geq 0$ must be a product. 
\end{proof}

We now prove our main result on Ricci solitons stated in the introduction.
\begin{proof}[Proof of Theorem \ref{thm_solitons}]

The first  claim  that $g_1$ and $g_2$ are metrics of the form (\ref{eqn_g_f(t)}) follows from Theorem \ref{thm_local_classification} and Corollary \ref{cor_soliton_f(t)}.

When $g_1$ is a complete shrinking or steady soliton, we also know from Proposition \ref{prop_nosplitting} that $g_1$ is either 
rotationally-symmetric (on $\mathbb{R}^n$ or $S^n$) or a product.  From the work of 
Kotschwar \cite{Kotschwar08} and Bryant \cite{Bryant}, the only complete  rotationally-symmetric gradient 
Ricci solitons with $\lambda_1 \geq 0$ are the round sphere, flat $\mathbb{R}^n$,  and the Bryant soliton.  In the flat case 
there are two  rotationally-symmetric gradient Ricci soliton structures with $f=f(s)$  on $g_1= ds^2 + s^2 g_{S^{n-1}}$:  
one where  $f$ is constant  and $\lambda_1 = 0$ and the other where $f$ is the Gaussian density,  
$f = \frac{\lambda_1}{2} s^2 + b$.   We will refer to the solitons in the first case as  flat Euclidean solitons 
and to the second case as  flat Gaussian solitons. 

   If $g_1$ is a product 
$\mathbb{R} \times N$ and $f_1=f_1(t)$, we have that 
$\mathrm{Hess}  f_1 = f_1'' dt^2$, so that $g_N$ must be Einstein.    By \cite{Ivey} any non-trivial compact gradient 
Ricci soliton is shrinking, so the compact result follows from the complete one. (In the trivial case
in which $f_1$ is constant, we can appeal to Theorem \ref{thm_compact_case}.)

The next claim, that if $g_2$ is also a soliton, then both spaces are round spheres or $\phi$ is stereographic projection, 
appears at the end of the section as Corollary \ref{cor_solitons}.

Finally, we prove that a complete gradient Ricci soliton $(M,g,f)$ admitting a 
non-homothetic, structure-preserving conformal field $V$ is Einstein with $f$ constant.
If $M$ is compact, then the first part of the proof, applied to the flow of $V$, implies that $f$ is constant.  
Thus, we assume $M$ is non-compact and $f$ is non-constant and appeal to the classification
derived in section \ref{sec_conf_fields_f(t)}. Since we consider a complete Ricci soliton, only the $f=f(t)$ case occurs by Corollary \ref{cor_soliton_f(t)}.
There are a couple cases to consider, in which $f$ and $\sigma$ are known explicitly.

Suppose $\omega=0$. By translating $t$ and rescaling $\sigma$, we have that 
$\sigma(t) = \frac{1}{\kappa}\sinh(\kappa t)$,
$f' = \frac{c}{A \cosh(\kappa t)}$ for nonzero constants $\kappa,c$ and $A$, and 
\[h = dt^2 + \cosh^2(\kappa t) g_N.\]
We compute $\mathrm{Ric}_h\left(\frac{\partial}{\partial t},\frac{\partial}{\partial t}\right)$ 
using both (\ref{eqn_Ric_conformal}) and (\ref{eqn_Ric_h_WP}) to show that
\[\lambda = e^{-\frac{2f}{n-2}} \left( -\kappa^2(n-1) - \frac{1}{n-2}(f')^2\right)- \frac{1}{n-2} (\Delta_g f - |\nabla f|^2_g),\]
where $g=e^{\frac{2f}{n-2}} h$ satisfies $\mathrm{Ric}_g + \mathrm{Hess}_g f = \lambda g$.
Next, using the conformal relation between $g$ and $h$, we find
\[\Delta_g f - |\nabla f|^2_g = e^{-\frac{2f}{n-2}} \Delta_h f,\]
and, by computing the Laplacian on a warped product, 
\[\Delta_h f = f'' + \frac{\kappa(n-1)f'\sinh(\kappa t)}{\cosh(\kappa t)}.\]
Thus,
\[\lambda = e^{-\frac{2f}{n-2}} \left( -\kappa^2(n-1) - \frac{1}{n-2}\left(f'' + (f')^2
+\frac{\kappa(n-1)f' \sinh(\kappa t)}{\cosh(\kappa t)}\right) \right).\]
Elementary analysis shows that $\lambda$ is non-constant.

Finally, suppose $\omega \neq 0$.  In this case, $g_N$ admits a homothetic field that is not Killing,
so $g_N$ is flat. Working in the variable $r = \int_0^t \frac{dt}{\sigma'(t)}$, we have
\[h = dt^2 + \sigma'(t)^2g_N = \sigma'(t)^2 (dr^2 + g_N).\]
The metric $g$ is given by
\[g=e^{\frac{2f(r)}{n-2}} \sigma'(r)(dr^2 + g_N).\]
Let $\varphi = \frac{f}{n-2} + \frac{1}{2} \log \sigma'(r)$, so
that $g=e^{2\varphi} (dr^2+g_N)$.
We use this conformal relation to find:
\begin{align*}
\mathrm{Ric}_g\left(\frac{\partial}{\partial r}, \frac{\partial}{\partial r}\right) &= -(n-1)\varphi''\\
\mathrm{Hess}_gf \left(\frac{\partial}{\partial r},\frac{\partial}{\partial r}\right) &= f''-\varphi'f',
\end{align*}
where all derivatives are with respect to $r$. In particular, if $\mathrm{Ric}_g + \mathrm{Hess}_g f=\lambda g$, then
\[\lambda =-(n-1)\varphi'' + f''-\varphi'f'.\]
Using $f'(r)=\frac{c}{r}$, straightforward computations show
\[\lambda = -\frac{n-1}{2}\left(\frac{\sigma'''}{\sigma'} - \frac{(\sigma'')^2}{(\sigma')^2}\right)
      -\frac{c\sigma''}{2r\sigma'} - \frac{c(c-1)}{(n-2)r^2}.\]
If $B=0$, we have $\sigma(r)= \frac{1}{2} - \frac{1}{\ln (Cr)}$.  If
$B>0$, then $\sigma(r) = \frac{1}{2} + \sqrt{B}\tan(\sqrt{B} \ln(Cr))$.
Elementary analysis shows that $\lambda$ is not constant in either case.

We conclude that $f$ must in fact be constant, so that $(M,g)$ is Einstein.  
\end{proof}

\begin{remark}
\label{rmk_kanai}
As an addendum to the proof of the last statement: Kanai showed a complete Einstein space admitting a non-homothetic 
conformal field belongs to the 
following list, up to rescaling \cite{Kanai} (cf. Theorem 2.7 of \cite{KR09}): a round sphere, Euclidean space, 
hyperbolic space, a warped product $ds^2 + e^{2s}g_N$ (where $N$ is Ricci-flat), or a warped product 
$ds^2 + \cosh^2(s) g_N$ (where $N$ has Einstein constant $-(n-2)$).
\end{remark}

We close with examples of Ricci solitons $g_1$ admitting structure-preserving conformal changes
to GQE metrics $g_2$, making use of Corollary \ref{cor_GQE_conformal}.  
These examples will be used in the proof of Corollary \ref{cor_solitons}.
\begin{example} [Product soliton]
We consider the case in which $g_1 = ds^2 + g_N$, where $g_N$ is Einstein with Einstein constant $\lambda_1$.  
$f_1=f_1(s)$ is necessarily of the form 
\[ f_1(s) = \frac{\lambda_1}{2} s^2 + a s + b, \]
for some constants $a$ and $b$.  Assume $f_1$ is not constant.

First we consider the case $\lambda_1 = 0$.  By adding a constant to $f_1$, we may assume $f_1(s) = as$, $a\neq 0$. Then we have
\begin{equation*}
h_1 = e^{\frac{-2as}{n-2}} g_1 = dt^2 + e^{\frac{-2as}{n-2}}  g_N, 
\end{equation*}
where $dt =e^{\frac{-as}{n-2}} ds$. 
$u$ is a solution to $\frac{du}{dt} = e^{\frac{-as}{n-2}}$, which implies that 
$\frac{du}{ds} = e^{\frac{-2as}{n-2}}$.
So $u(s) = -\frac{n-2}{2a}e^{\frac{-2as}{n-2}} + C$ and  we have 
\begin{equation*}
g_2  = K u^{-2} g_1= K \left( d\tau^2  + u^{-2} g_N \right), 
\end{equation*}
where $K$ is a positive constant and $\tau(s) = \int u(s)^{-1} ds$. 

If $a$ and $C$ have different signs, then $|u|>0$ for all $s$, giving a global conformal change $g_2 = Ku^{-2} g_1$.
Note that $\tau$ is always either bounded above or below (depending on the sign of $a$), so $g_2$ is not complete.  
If $a$ and $C$ have the same sign, then $u$ has a zero, and the conformal change is not global.

In the case $\lambda_1 \neq 0$, by shifting $s$ and adding a constant to $f_1$ we can assume that 
$f_1(s) = \frac{\lambda_1}{2} s^2$.  Then we have 
\begin{equation*}
h_1 = e^{\frac{-\lambda_1 s^2}{n-2}} g_1 = dt^2 + e^{\frac{-\lambda_1 s^2}{n-2}}  g_N,
\end{equation*}
where $dt =e^{\frac{-\lambda_1 s^2}{2(n-2)}} ds$. 
We have  $\frac{du}{ds} = e^{\frac{-\lambda_1 s^2}{n-2}}$, 
so that  $u(s) = C + \int_0^s e^{\frac{-\lambda_1 p^2}{n-2}} dp $ 
and  $g_2 =  K u^{-2} g_1$.  By Corollary \ref{cor_GQE_conformal}, $g_2$ is a gradient Ricci almost soliton with 
potential $f=f_1$.

Note that when $\lambda_1>0$, $u(s)$ is bounded  which implies that we can choose $C$ large enough so that $u$ does not 
vanish and that  $g_2$ is complete if $g_1$ is.  When $\lambda_1 <0$, $u$ will always have a zero, so there is 
no global conformal change. 

In the case $\lambda_1 \geq 0$, we point out that $(g_2,f)$ is not a gradient
Ricci soliton.  To see this, we note $\mathrm{Hess}_{g_1} u = u''(s) ds^2$ and
$\Delta_{g_1} u = u''(s)$ and compute
(where prime denotes a derivative with respect to $s$):
\[\mathrm{Ric}_{g_2} = (n-1)\left(\frac{u''}{u} -
\frac{(u')^2}{u^2}\right)ds^2 + \left(\lambda_1 + \frac{u''}{u} -
(n-1)\frac{(u')^2}{u^2}\right)g_N\]
and
\[\mathrm{Hess}_{g_2} f = \left(f'' + \frac{f'u'}{u}\right) ds^2 - \frac{f' u'}{u} g_N.\]
In order for $\mathrm{Ric}_{g_2} + \mathrm{Hess}_{g_2} f$ to equal
$\lambda_2 g_2$, we must have:
\[\lambda_2=(n-1)\left(uu'' - (u')^2\right) + \left(\lambda_1 u^2 + \lambda_1 s
uu'\right)\]
in the case $\lambda_1 > 0$, and 
\[\lambda_2=(n-1)\left(uu'' - (u')^2\right) + auu'\]
in the case $\lambda_1 = 0$.  However, in either case, one can explicitly show that $\lambda_2$ is not constant.
\end{example}

\begin{example}[Bryant Soliton]
The Bryant soliton is the unique (up to rescaling) rotationally-symmetric, steady, gradient Ricci soliton.  We write 
this metric as 
\[ g_1 = ds^2 + w(s)^2 g_{S^{n-1}}\]
for $s \geq 0$, where $w(0)=0$, $w'(0)=1$, and $w(s)>0$ for $s>0$.  We also have $w = O(s^{1/2})$, $w' = O(s^{-1/2})$, $w'' = O(s^{-3/2})$,
 the scalar curvature $R$ is $O(s^{-1})$ for $s$ large, and the sectional curvature is everywhere positive
 (see \cite{Bryant} or Chapter 1, section 4 of \cite{Chowetc}).  From (\ref{Sol_int}) we have 
\[ R + |\nabla f|^2 = c. \]
for some positive constant $c$. Thus, $f' \rightarrow \pm \sqrt{c}$ at infinity, so  $f = O(s)$.  Since $g_1$ has positive curvature,  $\mathrm{Hess} f$ is negative-definite
and we conclude $f' \rightarrow -\sqrt{c}$ at infinity. 
Now,
\[ h_1 = dt^2 + \left( e^{\frac{-f}{n-2}} w\right)^2 g_{S^{n-1}},\]
where $\frac{dt}{ds} = e^{-\frac{f}{n-2}}$, so that 
\[ u(s) = C + \int_0^s e^{\frac{-2f(p)}{n-2}} w(p) dp.\]
Since $s\geq 0$, from the asymptotics of $f$ and $w$ we see that $u$ blows up exponentially in $s$.
Thus we have a global conformal change to an incomplete metric $g_2$, provided $C> 0$.

Finally, we ask whether $g_2=u^{-2} g_1$ is also a Ricci soliton with potential $f$.
Assume $\mathrm{Ric}_{g_2} + \mathrm{Hess}_{g_2} f = \lambda_2 g_2$.  Direct calculation of the $ds^2$ component
of this equation leads to:
\begin{equation} \label{eqn:Rotsym} (n-1)\left(\frac{u''}{u}-\frac{w''}{w}+\frac{u'w'}{uw}-\frac{(u')^2}{u^2}\right) + f'' + \frac{f'u'}{u} = \lambda_2 u^{-2},\end{equation}
where all derivatives are with respect to $s$. Using $\mathrm{Ric}_{g_1} + \mathrm{Hess}_{g_1} f = 0$, we have
$f'' = (n-1)\frac{w''}{w}$.  This simplification leads to:
\[(n-1)(uu''+ uu'w'w^{-1} - (u')^2) + uu'f' = \lambda_2.\]
We show $\lambda_2$ is not constant by examining its asymptotics at $s=0$ and $s \to \infty$.  As $s \to 0^+$:
$f$ limits to $0$ (without loss of generality), $f'$ limits to $0$, $u$ limits to $C$, $u'$ limits to zero, and $u''$ limits to 
$1$.  This implies $\lim_{s \to 0^+} \nu(s) = 2C(n-1)$.
On the other hand, by the above asymptotics on $w$ and $f$, one can show that $\frac{u''}{u}$ and $\frac{u'}{u}$
limit to $\frac{4c}{(n-2)^2}$ and $\frac{2\sqrt{c}}{n-2}$ at infinity, respectively.  It follows that
$\lim_{s \to \infty} \nu(s) u(s)^{-2} = \frac{-2c}{n-2}$,
so that $\lambda_2(s) \to -\infty$ as $s \to \infty$.  In particular, $\lambda_2$ is not constant.
\end{example}

\begin{example}[Flat Gaussian Soliton]
There is one more example of rotationally-symmetric shrinking soliton: the flat Gaussian. In this case the metric is flat 
$\mathbb{R}^n$ written in polar coordinates as 
\[ g_1 = ds^2 + s^2 g_{S^{n-1}} \]
with $f = \frac{\lambda_1}{2}s^2 + b$, $\lambda_1 \neq 0$.  Without loss of generality we assume $b=0$. Then we have 
\[ h_1 = dt^2 + \left( e^{\frac{-\lambda_1 s^2}{2(n-2)}} s\right)^2 g_{S^{n-1}},\]
where $\frac{dt}{ds} = e^{\frac{-\lambda_1 s^2}{2(n-2)}}$, so  $\frac{du}{ds} = se^{\frac{-\lambda_1 s^2}{n-2}}$ and thus
\[ u(s) = C - \frac{(n-2)}{2\lambda_1} e^{\frac{-\lambda_1 s^2}{n-2}}.\]

Considering $g_2 = u^{-2} g_1$, note that when $\lambda_1>0$, $u(s)$ is bounded  which implies that we can choose $C$ large enough so that $u$ does not 
vanish and that $g_2$ is complete.  When $\lambda_1 <0$ and $C>0$ we also obtain a global conformal change, however $g_2$ will be incomplete.  

Finally we determine whether $g_2$ is a Ricci soliton. We know that $\mathrm{Ric}_{g_2} + \mathrm{Hess}_{g_2} f= \lambda_2 g_2$ for a function $\lambda_2$.   Arguing as in the Bryant soliton  example, by equation (\ref{eqn:Rotsym}) we obtain  
\[ (n-1)\left( uu'' + u'u s^{-1} - (u')^2 \right) + \lambda_1 + \lambda_1 s u'u = \lambda_2.\]
Then one can explicitly show that $\lambda_2$ is not constant in this case as well. 

\end{example}

Finally, we prove the following corollary, which completes the proof of Theorem \ref{thm_solitons}.
\begin{corollary}
\label{cor_solitons}
Let $\phi$ be a non-homothetic conformal diffeomorphism between Ricci solitons
$(M_1,g_1,f_1)$ and $(M_2,g_2,f_2)$ such that $\phi^* df_2 = df_1$.  If $(M_1,g_1)$ is complete and either shrinking or steady,
then $f_1$ and $f_2$ are constant, and  either $(M_1, g_1)$ and $(M_2, g_2)$ are both isometric to round spheres, 
or $\phi$ is a stereographic projection  with $(M_1,g_1)$ flat Euclidean space and  $(M_2,g_2)$  a 
round spherical metric with a point removed.
\end{corollary}
\begin{proof}
By Theorem \ref{thm_solitons}, we have that $(M_1, g_1, f_1)$ is a product of $\mathbb{R}$ with an Einstein manifold,
the Bryant soliton, a flat Gaussian soliton, a flat Euclidean space, or a round sphere.  However, the previous examples 
show the conformal transformations associated to the first three cases do not produce a soliton metric.  In the last two 
cases, $f$ is constant so we are in the Einstein case.  From Example \ref{ex_sphere} we can see the only time we have a 
global non-homothetic conformal diffeomorphism  from  a round spherical metric $g_1$ to another Einstein metric $g_2$ is 
when $g_2$ is also a round spherical metric.  A similar analysis shows that the only time we have a global non-homothetic 
conformal diffeomorphism from flat Euclidean space $g_1$ to another Einstein metric is the case of stereographic 
projection where $g_2$ is a round spherical metric with a point removed. 
\end{proof}

\appendix

\section{Conformal Fields on warped products over a one-dimensional base}
\label{app_conf_fields}
Here we collect some calculations for conformal fields of a Riemannian metric $h$ of the form 
\[ h = dt^2 + u(t)^2 g_N\]
on $M=I \times N$, where $I$ is an open interval. Let $V$ be a vector field on $M$.
We  write
\[ V = v_0(t, x) \frac{\partial}{\partial t} +V_t, \] 
where $v_0$ is a function on $M$ and $V_t$ is the projection of $V$ onto the factor $\{t\} \times N$.    
We have the following necessary and sufficient conditions for $V$ to be a conformal field  for $h$.

\begin{proposition}
\label{prop_conf_fields_app}
$V$ is a conformal field for $h$, 
\[ L_V h = 2 \sigma h,\]
if and only if 
\begin{enumerate}
\item $V_t$ is a conformal field for $g_N$ for each $t$: $L_{V_t} g_N = 2 \omega_t g_N$. 
\item  $ \frac{\partial}{\partial t}(v_0 u^{-1})  = \omega_t u^{-1}$
\item $\frac{\partial V_t}{\partial t} = -u^{-2} \nabla^N v_0$,
\end{enumerate}
where $\nabla^N v_0$ is the gradient of $v_0(t,\cdot)$ on $\{t\}\times N$.
Moreover $\sigma= v_0u^{-1} \frac{\partial u}{\partial t}  + \omega_t = \frac{\partial v_0}{\partial t}.$
\end{proposition}

\begin{proof} We compute the Lie derivative of $h$.  Let $(x^1, \ldots, x^{n-1})$ be normal coordinates at some 
$p \in N$, and let $V_t = v_i(t,x) \partial_i$, with the Einstein summation convention in effect for $i=1$ to $n-1$.
Here, $\partial_i = \frac{\partial}{\partial x^i}$ and we let $\partial_t = \frac{\partial}{\partial t}$. 
To begin, we record the following Lie brackets:
\begin{align*}
[V, \partial_t] &= -\frac{\partial v_0}{\partial t} \partial_t - \frac{\partial v_i}{\partial t} \partial_i\\
[V, \partial_j] &= -\frac{\partial v_0}{\partial x^j} \partial_t - \frac{\partial v_i}{\partial x^j} \partial_i.
\end{align*}
Now, at the point $(t,p)$,
\begin{align*}
(L_V h) (\partial_t, \partial_t) &= D_V h(\partial_t, \partial_t) -2 h([V,\partial_t],\partial_t)\\
  &= 2 \frac{\partial v_0}{\partial_t},\\
(L_V h) (\partial_t, \partial_j) &= D_V h(\partial_t, \partial_j)- h([V,\partial_t], \partial_j) - h(\partial_t, [V, \partial_j])\\
  &= u^2 \frac{\partial v_j}{\partial t} + \frac{\partial v_0}{\partial x^j},\\
(L_V h) (\partial_j, \partial_k) &= D_V h(\partial_j, \partial_k)- h([V,\partial_j], \partial_k) - h(\partial_j, [V, \partial_k])\\
  &= 2v_0 uu' \delta_{jk}+ u^2\left(\frac{\partial v_k}{\partial x^j}+\frac{\partial v_j}{\partial x^k} \right)\\
  &= 2v_0 uu' g_N(\partial_j, \partial_k)+u^2 (L_{V_t} g_N) (\partial_j, \partial_k).
\end{align*}
In particular, for arbitrary vector fields $X,Y$ tangent to $\{t\} \times N$,
\begin{align*}
(L_V h) (\partial_t, X) &= u^2 g_N\left(X, \frac{\partial V_t}{\partial t}\right) + D_X v_0,\\
(L_V h) (X,Y) &= 2v_0 uu' g_N(X,Y)+u^2 (L_{V_t} g_N) (X,Y)
\end{align*}
Then $L_V h$ equals $2 \sigma h$ if and only if
\begin{align*}
 \sigma &=  \frac{\partial v_0}{\partial t},\\
u^2 \frac{\partial V_t}{\partial t} &= - \nabla^N v_0,\\
L_{V_t} g_N &= 2 \omega_t g_N,
\end{align*}
where $\omega_t := \sigma - v_0 u^{-1} u'$.  The first equation is equivalent to:
\[ \frac{\partial}{\partial t}(v_0 u^{-1})  = \omega u^{-1}.\]
\end{proof}
Two consequences of this result are the following. 

\begin{corollary} \label{cor_v_0(t)}
If $V$ is a conformal field for $h$ as above, then $v_0 = v_0(t)$ if and only if $V_t$ is a fixed homothetic vector field for $g_N$. 
\end{corollary}
\begin{proof}
Equation (3) of the previous proposition shows that $v_0=v_0(t)$ if and only if $V_t$ is independent of $t$.  In this case, (2) implies that $\omega$ is constant. 
\end{proof}
In fact, we can solve for $v_0$ and $\sigma$ explicitly. 
\begin{corollary} \label{cor_fields_integrated_eqns} 
With notation as above, 
\begin{eqnarray*}
v_0 &=& u(t)  \left(A(x) + \int \frac{\omega_t(x)}{u(t)} dt \right)\\
\sigma &=&u'(t)\left(A(x) + \int \frac{\omega_t(x)}{u(t)} dt\right) + \omega_t(x) 
\end{eqnarray*}
where $A(x)$ is a function on $N$.
\end{corollary}
\begin{proof} 
Integrating equation (2) of the proposition with respect to $t$ gives the formula for $v_0$.  The formula for $\sigma$ follows from 
$\sigma = \frac{\partial v_0}{\partial t}.$
\end{proof}

\end{document}